\documentclass[12pt,a4paper]{amsart}

\usepackage{url,amssymb,amsmath,amsthm,amscd,paralist}
\usepackage{algpseudocode}
\usepackage{amsmath}
\usepackage{fancyhdr}
\usepackage{amssymb}
\usepackage{graphicx}
\usepackage[T1]{fontenc}
\usepackage{fullpage}
\usepackage{mathtools}
\usepackage{textcomp}
\usepackage{tikz-cd}

\usepackage[
backend=biber,
style=alphabetic,
sorting=nyt 
]{biblatex}

\newtheorem*{uundistorted}{Theorem 1.2}
\newtheorem*{mainmain}{Theorem 1.1}

\newtheorem{thm}{Theorem}[section]
\newtheorem{lemma}[thm]{Lemma}

\theoremstyle{definition}

\newtheorem*{thx}{Acknowledgments}

\newcommand{\Arc}{\mathcal{C}'(S)}

\newcommand{\ArcY}{\mathcal{C}'(Y)}
\newcommand{\ArcoY}{\mathcal{C}_0'(Y)}

\newcommand{\C}{\mathcal{C}(S)}
\newcommand{\Co}{\mathcal{C}_0(S)}
\newcommand{\CY}{\mathcal{C}(Y)}
\newcommand{\CoY}{\mathcal{C}_0(Y)}

\newcommand{\G}{\Gamma}
\newcommand{\Ghat}{\widehat{\Gamma}}

\newcommand{\M}{\mathcal{M}}
\newcommand{\MG}{\mathcal{M}_G}
\newcommand{\MM}{\mathcal{M}_{MM}}
\newcommand{\Mod}{\mbox{Mod}(S)}

\newcommand{\YM}{{Y_{MM}}}

\newcommand{\R}{\mathbb{R}}

\newcommand{\Z}{\mathbb{Z}}

\begin{document}

\author{Christopher Loa}

\title{Free products of Abelian groups in mapping class groups}
\maketitle

\begin{abstract}
    We construct a new family of examples of parabolically geometrically finite subgroups of the mapping class group in the sense of \cite{DDLS} and prove they are undistorted in $\Mod$.
\end{abstract}

\section{Introduction}

In \cite{FM} Farb and Mosher introduce a notion of \textit{convex cocompactness} for mapping class groups. The original notion of convex cocompactness comes from Kleinian groups, where it is a special case of geometric finiteness. In recent work \cite{DDLS}, Dowdall, Durham, Leininger, and Sisto have introduced a notion of \textit{parabolic geometric finiteness} for mapping class groups. Examples include convex cocompact groups (as one would hope) and finitely generated Veech groups by work of Tang \cite{T}. In this paper we construct a new family of examples of parabolically geometrically finite groups and prove they are undistorted in $\Mod$.

Let $S = S_{g,p}$ be an orientable surface with genus $g(S)$ and $0 \leq p(S) < \infty$ punctures. We measure the complexity of $S$ via $\xi(S) = 3g(S) + p(S) - 3$. Assume $\xi(S) \geq 1.$ Let $A, B$ be multicurves in $S$ and consider subgroups $H_A$ $\cong$ $\mathbb{Z}^n, H_B$ $\cong$ $\mathbb{Z}^m$ of Mod$(S)$ generated by multitwists about the components of $A$ and $B$, respectively. Let $G = \langle H_A, H_B \rangle$, and consider the natural homomorphism  $\Phi: H_A * H_B \to G$. Our first result says that if $A$ and $B$ are ``sufficiently far apart'' in $\C$, then  $\Phi$ is injective.

\begin{thm} \label{thm: main}
There exists a constant $D_0 \geq 3$, independent of S, with the following property. If $d_S (A,B) \geq D_0$, then  $\Phi: H_A * H_B \to G$ is injective, hence an isomorphism, and G is parabolically geometrically finite. Moreover, any element not conjugate into a factor is pseudo-Anosov.
\end{thm}

When $H_A, H_B$ are cyclic, we note the resemblance between Theorem 1.1 and \cite[Theorem 1.3]{Gu}. Notice that obtaining pseudo-Anosovs in this manner is a variation of Thurston's \cite{Th} and Penner's \cite{P} constructions of pseudo-Anosovs. Whereas Thurston's and Penner's constructions only require $A \cup B$ fill $S$, which can be assured by assuming $d_S(A,B) \geq 3$, the distance prescribed by Theorem 1.1 is relatively large. We therefore pose the question, given $H_A, H_B$, what is the optimal bound for $D_0$?

It is known that convex cocompact groups and finitely generated Veech groups are undistorted, i.e. quasiisometrically embedded, in $\Mod.$ Our second result is that our family of examples are also undistorted in $\Mod.$ In fact, this result holds for any parabolically geometrically finite $G = H_A*H_B$. That is, we needn't assume $A$ and $B$ are ``sufficiently far apart'' in $\C$ as in Theorem 1.1.

\begin{thm} \label{thm: undistorted} 
Let $G = H_A * H_B \subset$ \emph{Mod}$(S)$ be a nontrivial free product which is  parabolically geometrically finite with $H_A, H_B$ subgroups generated by multitwists about multicurves $A, B,$ respectively. Then $G$ is undistorted in \emph{Mod}$(S)$.
\end{thm}

Our results are complementary to recent results of Runnels \cite{R}. Given subgroups $H_A, H_B$ of $\Mod$ as described above, without any assumptions on their distance in $\C$,  \cite[Theorem 1.2]{R} tells you that, if the generators of $H_A, H_B$ are raised to sufficiently large powers, resulting in subgroups say $H_A',H_B'$, then $\langle H_A', H_B'\rangle$ is isomorphic to $H_A*H_B$ and undistorted in $\Mod.$ We note however that this is a special case and \cite[Theorem 1.2]{R} can be used to generate more general, undistorted RAAG's in $\Mod$. For more on embedding RAAG's in $\Mod$ see \cite{CLM, CMM, K, Seo}.

\subsection{Plan of the paper}
Section 2 is dedicated to background material. In section 3 we prove Theorem 1.1. In section 4 we introduce a \textit{general marking graph} and prove a Masur-Minsky style distance formula for it. Using a general marking graph as a model for $\Mod,$ we prove in section 5 that our subgroups are undistorted in $\Mod.$ In section 6, we discuss another notion of geometric finiteness and how our groups fit into this framework. We note that in each section we reindex the subscripts for variables.

\begin{thx}
I would like to thank my advisor Christopher J. Leininger for his guidance and support. I would also like to thank Marissa Miller and Jacob Russell for the helpful conversations regarding HHS structures and hieromorphisms.
\end{thx}

\section{Background}

\subsection{Coarse geometry}

Let $(X, d_X), (Y, d_Y)$ be metric spaces. We say that a map $f: X \rightarrow Y$ is a \textit{(K,C)-quasiisometric embedding} of $X$ into $Y,$ where $K \geq 1,  C \geq 0,$ if given any two points $a, b \in X$ we have

\begin{equation*}
    \frac{1}{K}d_X (a, b) - C \mbox{  } \leq \mbox{  } d_Y(f(a), f(b)) \mbox{  } \leq \mbox{  } Kd_X(a, b) + C.
\end{equation*}

If there exists $A>0$ such that $f(X)$ is $A$-dense in $Y$ we say that $f$ is a \textit{(K,C)-quasiisometry}.

Let $(X, d)$ be a metric space, $I\subset \R$ be an interval, and $f:I \to X$. We say that $f$ is a \textit{(K, C)-quasigeodesic} if it is a $(K,C)$-quasiisometric embedding. We say that $f$ is a \textit{D-local (K,C)-quasigeodesic} if $f|_J: J \rightarrow X$ is a $(K,C)$-quasiisometric embedding for all subintervals $J \subset I$ of length at most $D$.

Given two nonnegative real numbers $A, B,$ we say that $A$ and $B$ are \textit{(K,C)-comparable} for $K \geq 1, C \geq 0$ and write

\begin{equation*}
   A \asymp_{K,C} B 
\end{equation*}
if

\begin{equation*}
    \frac{1}{K}A - C \leq B \leq KA + C.
\end{equation*}

Note that this is not an equivalence relation since symmetry and transitivity fail. However
\begin{equation*}
    A \asymp_{K,C} B \Rightarrow B \asymp_{K, KC} A,
\end{equation*}
and 
\begin{equation*}
    A \asymp_{K,C} B \asymp_{K',C'} D \Rightarrow A \asymp_{KK',C' + \frac{C}{K'}} D.
\end{equation*}

We define $A \preceq_{K, C} B$ to mean

\begin{equation*}
    A \leq KB + C.
\end{equation*}
and finally

\begin{equation*}
  [A]_B =
  \begin{cases}
    A & \text{if $A \geq B$} \\
    0 & \text{otherwise}.
  \end{cases}
\end{equation*}

In general, two quantities being comparable does not imply that the truncated quantities are comparable. However, we do have the following:

\begin{lemma}
Let $\{x_i\}_{i=1}^N, \{y_i\}_{i=1}^N$ be two finite sequences of nonnegative numbers such that 

\begin{equation*}
    x_i \asymp_{K,C} y_i,
\end{equation*}
with $K \geq 1, C \geq 0$. If $\kappa > 2KC$ then

\begin{equation*}
    \sum_{i=1}^N [{x_i}]_\kappa \preceq_{2K, 0} \sum_{i=1}^N [y_i]_C.
\end{equation*}
\end{lemma}

\begin{proof}

Define

\begin{equation*}
    \Omega = \{i | x_i \geq \kappa \}. 
\end{equation*}

Then for $i \in \Omega$

\begin{equation*}
    x_i > 2KC > 2C \Rightarrow \frac{1}{2} \geq \frac{1}{2K} > \frac{C}{x_i}.
\end{equation*}

By comparability we have

\begin{equation*}
    \frac{x_i}{K} - C \leq y_i \leq Kx_i + C.
\end{equation*}

We divide through by $x_i$ which gives for $i \in \Omega$

\begin{equation*}
    \frac{1}{2K} < \frac{1}{K} -\frac{C}{x_i}\leq \frac{y_i}{x_i} \leq K + \frac{C}{x_i} < K + \frac{1}{2} < 2K,
\end{equation*}
i.e.

\begin{equation*}
    x_i \asymp_{2K, 0} y_i.
\end{equation*}

Using the comparability lower bound we have

\begin{equation*}
    \sum_{i=1}^N [{x_i}]_\kappa = \sum_{i \in \Omega} x_i \leq 2K\sum_{i \in \Omega} y_i \leq 2K\sum_{i=1}^N [y_i]_C.
\end{equation*}
\end{proof}













\subsection{$\delta$-hyperbolicity}
Let $X$ be a geodesic metric space. We say a geodesic triangle is $\delta$-\textit{thin} if each side is contained in the $\delta$-neighborhood of the other two. We say that $X$ is $\delta$-\textit{hyperbolic} or \textit{Gromov hyperbolic} if all geodesic triangles are $\delta$-thin. $\delta$-hyperbolicity is a quasiisometric invariant. See \cite{BH} for a full treatment.

\subsection{The curve and arc graphs}

Let $(\C, d_S)$ denote the \textit{curve graph} of $S$. For $\xi(S) \geq 2$, the vertices are homotopy classes of essential simple closed curves on $S$ with an edge of length 1 between any vertices which correspond to curves with disjoint representatives. For $\xi(S) = 1$, $S$ is either a once-punctured torus or a 4-punctured sphere. We again define the vertices to be homotopy classes of essential simple closed curves, but with edges of length 1 between vertices which correspond to curves with representatives that intersect minimally. For the once-punctured torus this minimum intersection is 1 and for the 4-punctured sphere the minimum intersection is 2. It is a well-known fact that in this case $\C$ is isomorphic to the Farey graph.
We define $\C$ to be empty for $\xi(S) = 0$ (i.e. for a pair of pants and a torus). For $\xi(S) = -1$, $S$ is an annulus and we describe $\C$ in detail in a later subsection.

In \cite{MM1}, Masur and Minsky prove that $\C$ is $\delta$-hyperbolic and has infinite diameter for $\xi(S) \geq 1$. 
It turns out that for the annular case, $\C$ is quasiisometric to $\Z$, hence is also $\delta$-hyperbolic. The best known hyperbolicity constant $\delta$ for $\C$ is 17 \cite{HPW}.

Let $(\Arc, d_S')$ denote the \textit{arc and curve graph} of $S$ for $\xi(S) \geq 1.$
The vertices are both homotopy classes of essential simple closed curves and \textit{arcs}. An arc is a homotopy class of a properly embedded path in $S$ which, fixing the endpoints, cannot be deformed into a boundary component or puncture. 
Again, edges of length 1 are put between vertices with disjoint representatives.






\subsection{Relative hyperbolicity}

The notion of relative hyperbolicity was first introduced by Gromov in \cite{Gr}. Since then many equivalent definitions have been suggested. Throughout this paper it is useful to have the following two in mind.

The first is due to Farb \cite{F}. Let $G$ be a finitely generated group and $\G$ a Cayley graph for $G$ with respect to a finite generating set. Given a finite set of finitely generated subgroups $\{H_1, ..., H_k\}$, the \textit{coned-off Cayley graph} $\Ghat = \Ghat(G, \{H_1, ..., H_k\})$ is obtained by adding a vertex $v(gH_i)$ for every left coset and adding an edge of length $\frac{1}{2}$ from every element of $gH_i$ to $v(gH_i)$. We say that $(G, \{H_1, ..., H_k\})$ satisfies the \textit{bounded coset penetration} property (BCP for short) if paths in $\G$ that ``quotient'' to quasigeodesics in $\Ghat$ penetrate cosets similarly. See \cite{F} for a precise statement. We say that $G$ is \textit{hyperbolic relative to} $\{H_1, ..., H_k\}$ if $\Ghat(G, \{H_1, ..., H_k\})$ is Gromov hyperbolic and the BCP property is satisfied.

The second definition was proposed by Bowditch in \cite{Bow}. Again, let $G$ be a finitely generated group and $\{H_1, ..., H_k\}$ a finite collection of finitely generated subgroups. Let $\mathcal{H}$ denote the set of all $H_i$-conjugates and suppose $G$ acts on a connected graph $T$ such that the following hold:
\begin{enumerate}
    \item $T$ is hyperbolic and each edge of $T$ is contained in only finitely many circuits of length $n$ for any $n \in \mathbb{N}$. A \textit{circuit} is a closed path without repeated vertices.
    \item There are finitely many edge orbits $Ge$, and each edge stabilizer $S_e$ is finite.
    \item The elements of $\mathcal{H}$ are precisely the infinite vertex stabilizers $S_v.$
\end{enumerate}

Then we say that $G$ is hyperbolic relative to $\{H_1, ..., H_k\}$. Both of the above definitions are known to be equivalent \cite{Bow}.

\subsection{Parabolic geometric finiteness} In \cite{DDLS}, Dowdall, Durham, Leininger, and Sisto introduce a notion geometric finiteness for subgroups of $\Mod$, generalizing Farb and Mosher's notion of convex cocompactness \cite{FM} (see also \cite{H, KL}). We say $G \subset \Mod$ is \textit{parabolically geometrically finite} (PGF for short) if the following hold:

\begin{enumerate}
    \item $G$ is hyperbolic relative to a (possibly trivial) collection of subgroups $\{H_1, ..., H_k\}$ with each $H_i$ an Abelian group virtually generated by multitwists about a multicurve $A_i$.
    \item There is an  equivariant quasiisometric embedding $\phi: \Ghat \to \C$.
\end{enumerate}

Examples include convex cocompact groups by work of Hamenst{\"a}dt \cite{H} and Kent-Leininger \cite{KL}, and finitely generated Veech groups by work of Tang \cite{T}.

\subsection{Subsurface projections}

Subsurface projections are an important part of our discussion, so we will briefly outline their construction following \cite{MM2}. This construction is originally due to Ivanov. A payoff is the Masur and Minsky's bounded geodesic image theorem, which is analogous in $\C$ to geodesics having uniformly bounded closest-point projections to any disjoint horoball in $\mathbb{H}^3$.

Define a  \textit{domain Y} in $S$ be an isotopy class of a connected, incompressible, non-peripheral, open subsurface. We say $Y$ is a \textit{proper subdomain} if $Y \neq S$. We first assume $\xi(Y) > 0$, since matters are simpler and the initial intuition is more easily gleaned. A \textit{subsurface projection} is a map $\pi_Y: \Co \to \mathcal{P}(\CoY)$ (where $\mathcal{P}(X)$ is the set of finite subsets of $X$) constructed as follows.

First, let $\pi_Y' : \Co \to \mathcal{P}(\ArcoY)$
map a simple closed curve to the simplex in $\ArcoY$ whose vertices are homotopy classes of arcs which represent the curve's minimal intersection with $Y$, or just the curve itself if it is contained in $Y$. We say $\pi_Y'(\alpha) = \emptyset$ if $\alpha$ does not meet $Y$ essentially. By \cite[Lemma 2.2]{MM2}, $\CY$ embeds into the $\ArcY$ as a cobounded set. Specifically, Masur and Minksy show that there is a coarse retraction $\psi_Y: \ArcoY \to \mathcal{P}(\CoY)$ that perturbs points no more than distance 1 and distances between adjacent points no more than 2. Finally, we compose these maps to get a map $\pi_Y  = \psi_Y \circ \pi'_Y : \Co \to \mathcal{P}(\CoY)$. Given $\alpha, \beta \in \Co$ with $\pi_Y(\alpha), \pi_Y(\beta) \neq \emptyset$ we define $d_Y(\alpha, \beta) = \mbox{diam}_Y(\pi_Y(\alpha), \pi_Y(\beta)).$

Though we haven't defined subsurface projections for $\xi(Y)=-1$, we go ahead and state the bounded geodesic image theorem.

\begin{thm}[Thm 3.1 of \cite{MM2}]

There exists a constant $M(\xi(S))$ with the following property. Let Y be a proper subdomain of $S$ with $\xi(Y) \neq 0$. Let g be a geodesic segment, ray, or bi-infinite line in $\C$ such that $\pi_Y(v) \neq \emptyset$ for every vertex v in g.
Then
\begin{equation*}
    \mbox{\emph{diam}}_Y(g) \leq M
\end{equation*}

\end{thm}

In \cite{W}, Webb proves there exists a uniform $M$ for Theorem 2.2, independent of $S.$ An important property of subsurface projections is that they are Lipschitz:

\begin{thm}[Lemma 2.3 of \cite{MM2}]
Let $Y \subset S$ be a domain and let $\Delta$ be a simplex in $\C$ with $\pi_Y(\Delta) \neq \emptyset$. Then \emph{diam}$_Y(\Delta) \leq 2$. If Y is an annulus then  \emph{diam}$_Y(\Delta) \leq 1$.
\end{thm}

\subsection{Annular domains and projections}

Given an annular subsurface $Y$ in $S$, we say that $Y$ is an \textit{annular domain} if it has incompressible boundary and is not homotopic to a puncture. 

Let $Y$ be an annular domain with core curve $\gamma$ and let $\widehat{Y}$ be the compactification of the annular cover $\tilde{Y}$ of $S$ corresponding to $\gamma$ (i.e. the quotient of $\mathbb{H}^2 \cup S^1_\infty - \gamma^\pm$  by the action of the isometry corresponding to $\gamma,$ where $\gamma^\pm$ are the fixed points of said isometry). Let $\widehat{\gamma}$ be the core curve of $\widehat{Y}$.

\begin{figure}[ht]
    \centering
    \includegraphics[trim = 7.5cm 4cm 5cm 3.25cm, clip=true, scale=0.5]{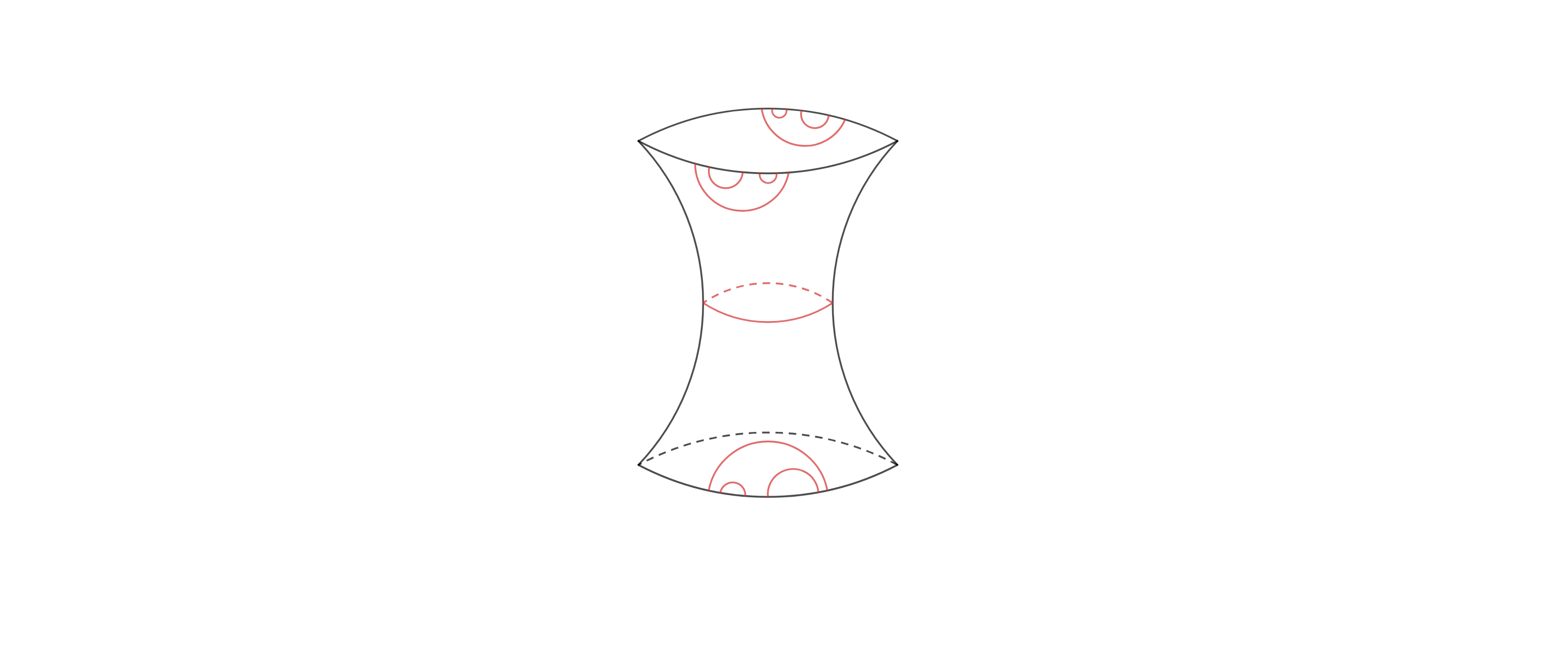}
    \caption{$\widehat{Y}$ with lifts of $\gamma$}
    \label{fig:Y^}
\end{figure}

We define the vertices of $\CY$ to be arcs connecting the two components of $\partial \widehat{Y}$ modulo homotopies fixing $\partial \widehat{Y}$ pointwise, and we add an edge of length 1 between any two vertices which have representatives with disjoint interiors.

Fix an orientation on $S$ and an ordering on the components of $\partial \widehat{Y}$. This allows us to define an algebraic intersection number $\alpha \cdot \beta$ for $\alpha, \beta \in \mathcal{C}_0(Y)$. One can show that

\begin{equation}
    d_Y(\alpha, \beta) = |\alpha \cdot \beta| + 1,
\end{equation}
for all $\alpha \neq \beta$ and

\begin{equation}
    \alpha \cdot \rho = \alpha \cdot \beta + \beta \cdot \rho + \varepsilon(\alpha, \beta, \rho),
\end{equation}
with $\varepsilon \in \{-1, 0, 1\}$ depending on the arrangement of the endpoints.

Using $(1), (2)$, we may construct a quasiisometry $f: \CY \to \Z$ by fixing $\rho \in \CoY$ and defining $f(\alpha) = \alpha \cdot \rho$. Then the identities give for all $\alpha, \beta \in \CoY$

\begin{equation*}
    |f(\alpha) - f(\beta)| \leq d_Y(\alpha, \beta) \leq |f(\alpha) - f(\beta)| + 2.
\end{equation*}

We now construct subsurface projections $\pi_Y: \Co \rightarrow \mathcal{P}(\CoY)$ for annuli. If $\alpha \in \Co$ does not intersect $Y$ essentially (notice this includes the core of  $Y$) then $\pi_Y(\alpha) = \emptyset$. If $\alpha \in \Co$ intersects $\gamma$ essentially, then at least one lift of $\alpha$ to $\widehat{Y}$ connects the components of $\partial \widehat{Y}$. Let $\pi_Y(\alpha)$ be the set of such lifts, which has diameter 1 in $\CY$.

Let $t_{\widehat{\gamma}}$ be the Dehn twist of $\widehat{Y}$ about its core $\widehat{\gamma}$. It follows from equation (1) 
that

\begin{equation*}
  d_Y(\alpha, t_{\widehat{\gamma}}^n(\alpha)) = |n|,
\end{equation*}
for all $\alpha \in \CoY$ and $n \in \mathbb{Z}$. Now notice that for $\beta \in \Co$ with $\pi_Y(\beta) \neq \emptyset$ we have 

\begin{equation}
   |n| - 5 \leq d_Y(\beta, t_\gamma^n(\beta)) \leq |n| + 5,
 \end{equation}
for all $n \in \Z.$ The reason for the coarseness of equation (3) is that the twist $t_\gamma$ of $S$ shifts each intersection of a lift of $\beta$ with all the lifts of $\gamma$. See figure 2.

\begin{figure}[h]
    \centering
    \includegraphics[trim = 7.5cm 3.5cm 8cm 3.1cm, clip = true, scale = .5]{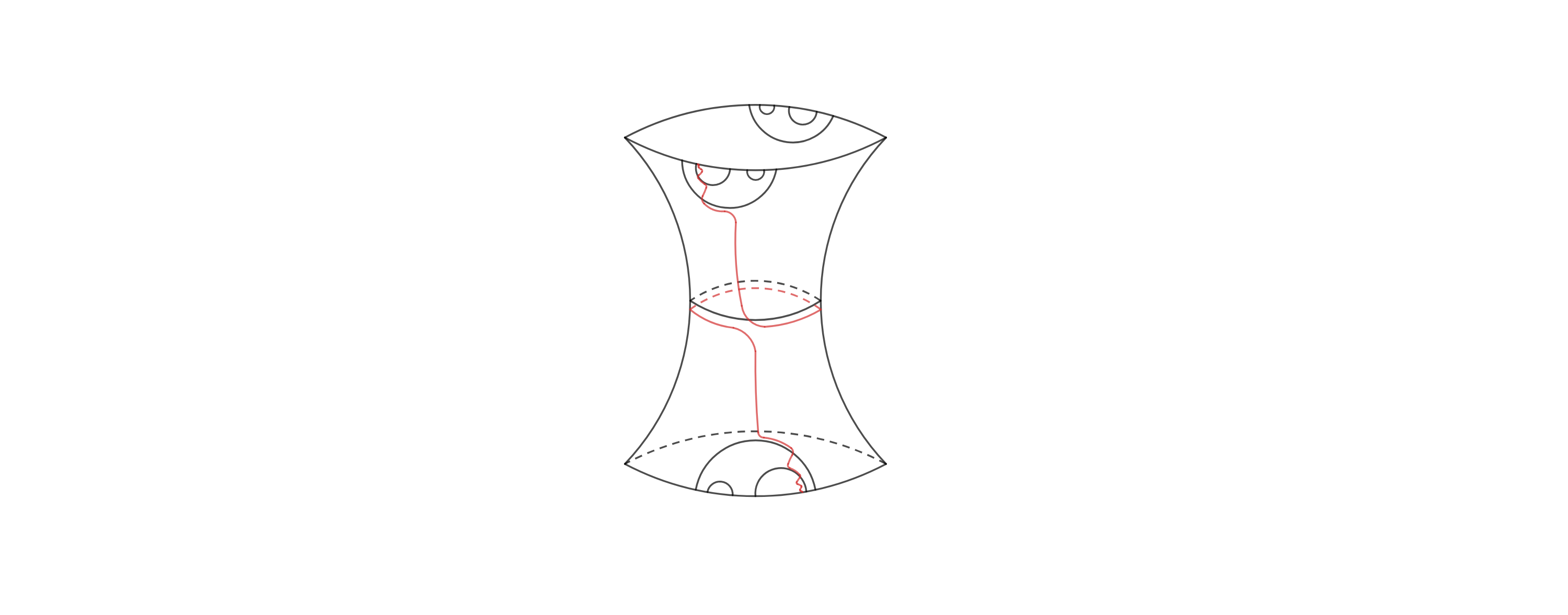}
    \caption{Lift of $t_\gamma (\beta)$}
    \label{fig:lift of dehn twist}
\end{figure}

\subsection{Disconnected domains.}

Let $A$ be a multicurve with components $\alpha_1, ... \alpha_k$ and let $Y_i$ be the annular domain corresponding to $\alpha_i$. We define

\begin{equation*}
   \mathcal{C}(A) = \mathcal{C}(Y_1) \times ... \times \mathcal{C}(Y_k),
\end{equation*}
and 

\begin{equation*}
   \pi_A = \pi_{Y_1} \times ... \times\pi_{Y_k}.
\end{equation*}

We define $d_A$ to be the $L^1$-metric. That is, for $\alpha, \beta \in \C$ such that $\pi_{Y_i}(\alpha), \pi_{Y_i}(\beta) \neq \emptyset$ for all $1 \leq i \leq k$ we have

\begin{equation*}
    d_A (\alpha, \beta) = \sum_{i=1}^k d_{Y_i}(\alpha, \beta).
\end{equation*}

Similarly define $(\mathcal{C}(Z), d_Z)$ with $Z$ a union of disjoint domains $Z_1, ..., Z_k \subset S.$

\section{Proving Theorem 1.1}

\subsection{Relative hyperbolicity of $H_A * H_B$} 
Let $X$ be a space composed of tori $T_A, T_B$ with dimensions equal to the ranks of $H_A, H_B$, respectively, connected by an edge $[0, 1]$ at points $t_A$, $t_B$ on $T_A, T_B$, respectively. Notice $X$ is a Salvetti complex for $H_A * H_B$ with a point blown up to an edge. Clearly $\pi_1(X) \cong H_A * H_B$. Fix a basepoint $x = \frac{1}{2}$. Let $p: (\widetilde{X}, \tilde{x}) \to (X,x)$ denote the universal cover. Notice each component of the preimage of $T_A, T_B$ corresponds to a left coset $gH_A, gH_B,$ respectively. Collapsing each of these components in $\widetilde{X}$ to points $ga, gb$ gives us a tree $T$. Let $W_A, W_B$ denote the set of all such $ga, gb,$ respectively and define $W = W_A \cup W_B$.

Let $\G$ be the Cayley graph of $H_A * H_B \cong \Z^n * \Z^m.$  Notice that $T$ is quasiisometric to $\Ghat(H_A * H_B, \{H_A, H_B\})$. Using the action of $H_A * H_B$ on $T$, it's easy to see that $H_A * H_B$ is hyperbolic relative to $\{H_A, H_B\}$ via Bowditch's definition. Alternatively, it's also easy to see directly via Farb's definition that $H_A * H_B$ is hyperbolic relative to $\{H_A, H_B\}$.

\subsection{Quasiisometrically embedding $T$ into $\C$} Recall our first theorem:

\begin{mainmain}
\textit{There exists a constant $D_0 \geq 3$, independent of S, with the following property. If $d_S (A,B) \geq D_0$, then  $\Phi: H_A * H_B \to G$ is injective, hence an isomorphism, and G is PGF. Moreover, any element not conjugate into a factor is pseudo-Anosov.}
\end{mainmain}

Let $D = d_S(A, B) = \mbox{diam}_S(A,B)$. Let $\alpha_1, ..., \alpha_{n'}$ and $\beta_1, ..., \beta_{m'}$ denote the components of $A$ and $B$, respectively. Without loss of generality, we may assume that $d_S(\alpha_1, \beta_1) = D$. The choices of $\alpha_1$ and $\beta_1$ are unimportant, we are simply choosing vertices of the simplices $A,B$ in $\C$. We first map $T$ to the scaled up tree $T_D$, isomorphic to $T$ but with edges of length $D$, with the map defined linearly on edges. Call this map $\phi_D$, which is a $(D,0)$-quasiisometry. $H_A*H_B$ also acts on $T_D$ so $\phi_D$ is equivariant. We abuse notation and also let $W$ denote the vertices of $T_D$.

The natural homomorphism $\Phi: H_A * H_B \to G = \langle H_A, H_B \rangle \subset \Mod$ induces an action of $H_A * H_B$ on $\C$. Define $\phi': T_D \to \mathcal{C}(S)$ by $\phi'(ga) = \Phi(g)\alpha_1, \phi'(gb) = \Phi(g)\beta_1,$ and fixing once and for all a geodesic $[\alpha_1, \beta_1]$, we have isometric embeddings $\phi'(g[a,b]) = \Phi(g)[\alpha_1, \beta_1]$. We claim that $\phi'$ is a quasiisometric embedding if $D$ is sufficiently large. We show this by showing that unions of adjacent edges in $T_D$ map to $D$-local quasigeodesics. It is a well-known result of Coornaert, Delzant, Papadopoulos \cite{CDP} that $D$-local uniform quasigeodesics are global quasigeodesics for $D$ sufficiently large.

\begin{lemma}
Let $h \in H_B -\{1\}$. Then

\begin{equation*}
    d_S(\alpha_1, \Phi(h)(\alpha_1)) \geq \frac{2D- 4}{N},
\end{equation*}
where $N>M + 5$ and $M$ is the constant from Theorem 2.2.
\end{lemma}

\begin{proof} Let $H_B = \langle b_1, ..., b_m\rangle$. Here each $b_i = t^{x^1_i}_{\beta_1}...t^{x^{m'}_i}_{\beta_{m'}} $ is a multitwist about the components $\{\beta_1, \beta_2,..., \beta_{m'} \}$ of $B$. We have $\Phi(h) = {b'}_1^{N_1},..., {b'}_k^{N_k}$ with $b'_i \in \{ b_1,..., b_m\}$ and $N_i \in \Z$. Let $Y_i$ be the corresponding annular subsurface of $S$ corresponding to $\beta_i$. Since $h \neq 1,$ $\Phi(h)$ acts with positive translation length on some $Y_i$. Without loss of generality, assume $\Phi(h)$ acts with positive translation length on $Y = Y_2$. One might want to choose $Y = Y_1$, but we choose the annular subsurface corresponding to $\beta_2$ to emphasize that the following argument doesn't depend on our already chosen $\beta_1$ ($D = d_S (\alpha_1, \beta_1)$).

Let $x \in [\alpha_1, \beta_1] -$st$(\beta_2)$. We remove st$(\beta_2)$ to ensure $\pi_Y(x) \neq \emptyset$. It follows from equation (3) and the fact that $t_{\beta_i}$ acts with bounded orbit on $\CY$ for $i \neq 2$ that for $N \neq 0$ we have
\begin{equation*}
   |NN'| - 5 \leq d_Y (x, \Phi(h)^N(x)) \leq |NN'| + 5,
\end{equation*}
where $N'$ is the net number of $t_{\beta_2}$'s that appear in $\Phi(h)$. Choose $N$ sufficiently large so that $|NN'| - 5 > M.$ The contrapositive of Theorem 2.2 gives that any geodesic $[x, \Phi(h)^N(x)]$ must pass close to $\beta_2$ in $\C$. Specifically there is a vertex $\gamma \in [x, \Phi(h)^N(x)]$ such that $\pi_Y (\gamma) = \emptyset$ and hence $d(\beta_2, \gamma) \leq 1$. This gives us

\begin{figure}[h]
    \centering
    \includegraphics[trim = 50cm 20cm 50cm 18cm, clip=true, scale = .3]{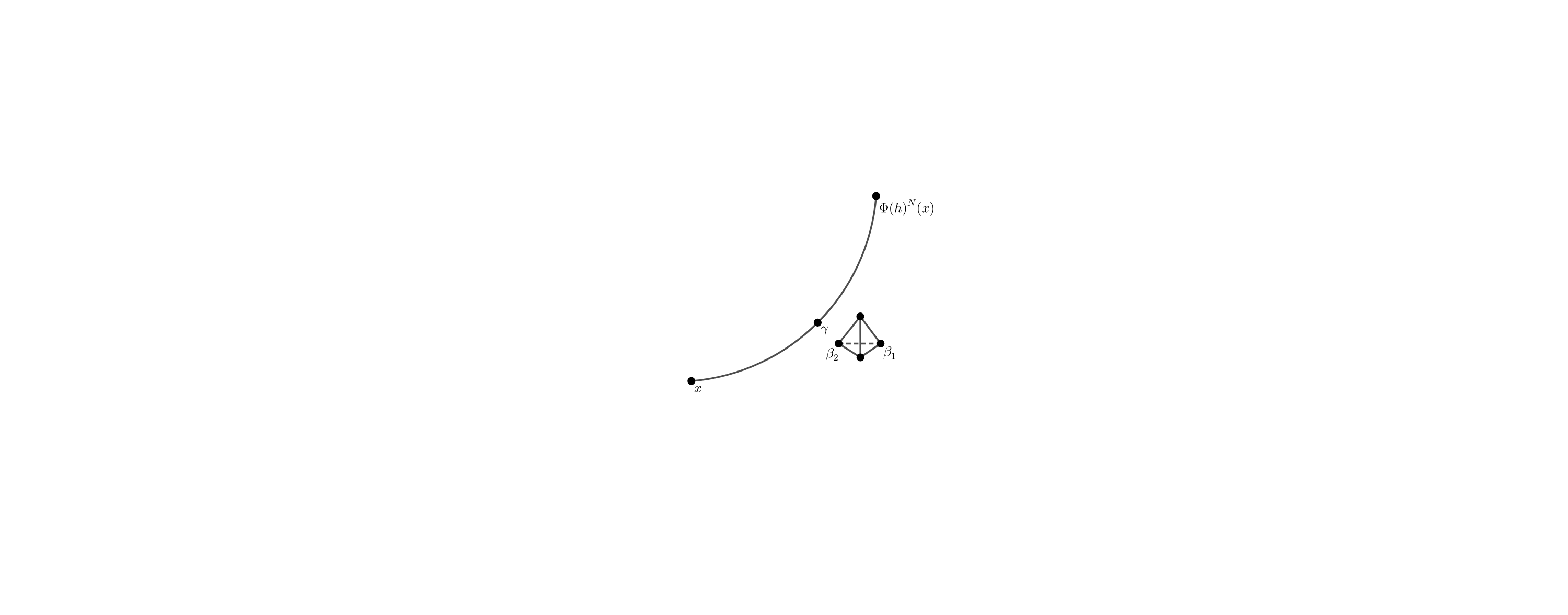}
    \caption{$d_S(x,\Phi(h)^N(x)) \geq 2d_S(x, \beta_1) - 4$}
    \label{fig:d(x, phi(h)^N(x))}
\end{figure}

\begin{equation*}
    d_S(x, \Phi(h)^N(x)) \geq 2d_S(x, \beta_1) -4.
\end{equation*}







The triangle inequality gives us

\begin{equation*}
    d_S(x, \Phi(h)(x)) \geq \frac{2d_S(x, \beta_1)- 4}{N}.
\end{equation*}

The same argument shows that taking $N > M + 5$ gives the above inequality for any $h \in H_B - \{1\}.$ Taking $x = \alpha_1$ gives us

\begin{equation*}
    d_S(\alpha_1, \Phi(h)(\alpha_1)) \geq \frac{2D- 4}{N}.
\end{equation*}
\end{proof}
We can get a better bound for $d(\alpha_1, \Phi(h)(\alpha_1))$ by using hyperbolicity of $\C$.

\begin{lemma}
Let $h \in H_B -\{1\}$. Then

\begin{equation*}
    d_S(\alpha_1, \Phi(h)(\alpha_1)) \geq 2D - 2((N+1)\delta +2),
\end{equation*}
where $N>M + 5$ and $M$ is the constant from Theorem 2.2. Here $\delta$ is the hyperbolicity constant of $\C.$
\end{lemma}

\begin{proof}
Since $\C$ is $\delta$-hyperbolic, any triangle $(\alpha_1, \beta_1, \Phi(h)(\alpha_1))$ is $\delta$-thin so we can find points $x \in [\alpha_1, \beta_1]$, $y \in [\beta_1, \Phi(h)(\alpha_1)]$, $z \in [\alpha_1, \Phi(h)(\alpha_1)]$ with $d_S(x,y), d_S(x,z) \leq \delta$. We'll show that $x$ is uniformly close to $\beta_1$ and hence $d_S(\alpha_1, \Phi(h)(\alpha_1))$ is roughly $2D$. See figure 4.

\begin{figure}[h]
    \centering
    \includegraphics[trim = 10.6cm 6.5cm 9cm 4cm, clip = true, scale = 1]{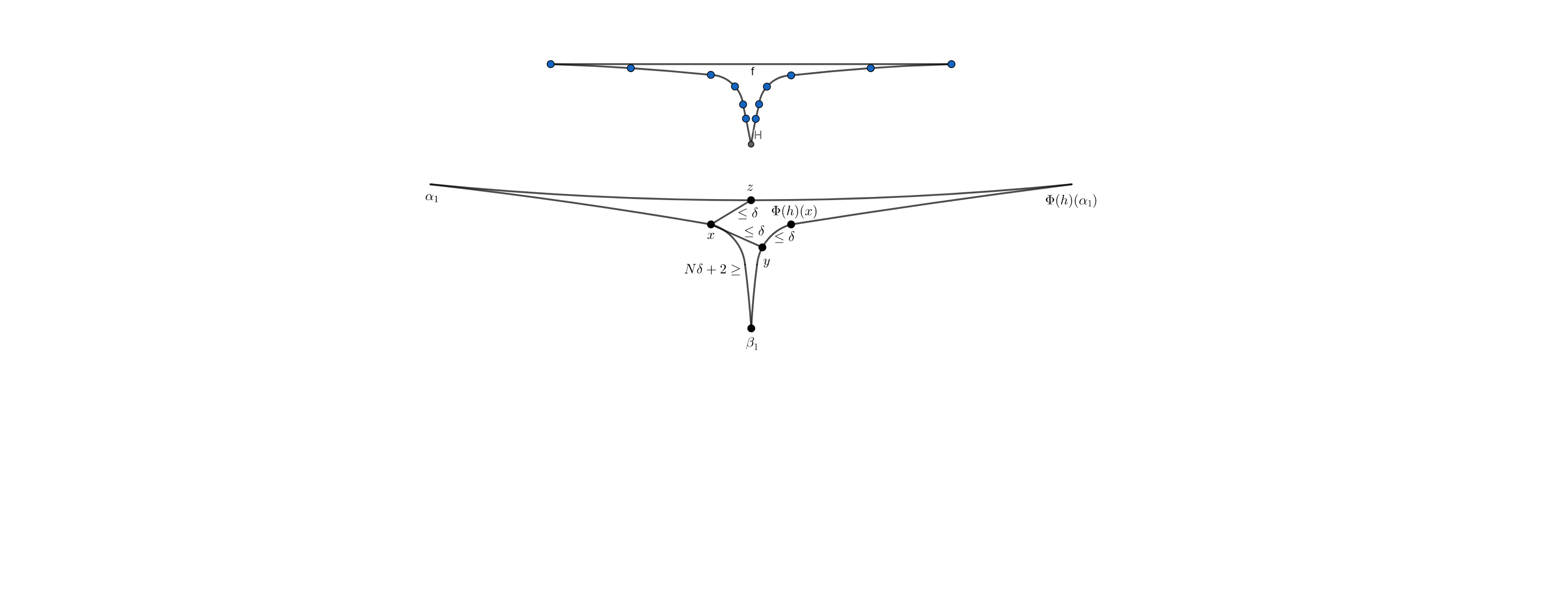}
    \caption{$d_S(\alpha_1, \Phi(h)(\alpha_1))$}
    \label{fig:d(a1, phi(h)(a1))}
\end{figure}

If $x \in$ st$(\beta_1)$ then we're done since this gives

\begin{equation*}
    d_S(\alpha_1, \Phi(h)(\alpha_1)) \geq 2d_S(\alpha_1, \beta_1) - 2(1 +  \delta),
\end{equation*}
since this implies $d_S(z, \beta_1) \leq 1 + \delta$ and $z \in [\alpha_1, \Phi(h)(\alpha_1)].$ Otherwise, the triangle inequality gives

\begin{equation*}
    d_S(x, \beta_1) - \delta \leq d_S(y, \beta_1) \leq d_S(x, \beta_1) + \delta.
\end{equation*}

Since $d_S(x, \beta_1) = d_S(\Phi(h)(x), \beta_1)$ and $y, \Phi(h)(x) \in [\beta_1, \Phi(h)(\alpha_1)]$ we get

\begin{equation*}
    d_S(y, \Phi(h)(x)) \leq \delta.
\end{equation*}

Then by the triangle inequality we have

\begin{equation*}
    d_S(x, \Phi(h)(x)) \leq 2\delta.
\end{equation*}

As in the proof of Lemma 3.1 we get


\begin{equation*}
    d_S(x,\beta_1) \leq N\delta + 2.
\end{equation*}

Hence 

\begin{equation*}
   d_S(z, \beta_1) \leq (N + 1)\delta + 2,
\end{equation*}
and since $z \in [\alpha_1, h(\alpha_1)]$ this gives 
\begin{equation*}
    d_S(\alpha_1, \Phi(h)(\alpha_1)) \geq 2D - 2((N+1)\delta +2).
\end{equation*}
\end{proof}

So we can indeed ensure that $d_S(\alpha_1, \Phi(h)(\alpha_1))$ is as large as we want by making $D = d_S(A, B)$ large. The same argument gives $d_S(\beta_1, \Phi(g)(\beta_1)) \geq 2D - 2((N+1)\delta +2)$ for any $g \in H_A - \{1\}.$ Hence the adjacent edges are mapped uniformly close to geodesics. We now prove that adjacent edges map to uniform quasigeodesics.

\begin{lemma}
The map $\phi': T_D \to \C$ maps geodesics to $D$-local $(1, 13\delta + 2C_0)$-quasigeodesics where $C_0 = (N+1)\delta  +2$, $N > M + 5$, and $M$ is the constant from Theorem 2.2.
\end{lemma}

\begin{proof}
Since we are only concerned about local behaviour, given any geodesic in $T_D$, it suffices to show that any adjacent edges map to a $(1, 13\delta + 2C_0)$-quasigeodesic segment. By translating and interchanging the roles of $H_A$ and $H_B$ it suffices to show that for any $x_0 \in [a,b]$, $y_0 \in [b, ha]$ with $h \in H_B - \{1\}$ the following inequality holds

\begin{equation*}
    d_{T_D}(x_0,y_0) - 13\delta - 2C_0 \leq d_S(x, y) \leq d_{T_D}(x_0,y_0),
\end{equation*}
where $x = \phi'(x_0), y= \phi'(y_0)$. The upper bound is automatic by the triangle inequality and the definition of $\phi'$, so we prove the lower bound. The points $x,y$ lie on the $\delta$-thin triangle $(\alpha_1, \beta_1, \Phi(h)(\alpha_1))$. Specifically $x \in [\alpha_1, \beta_1], y \in [\beta_1, \Phi(h)(\alpha_1)]$. 

\textbf{Case 1}: $d_S(x, \beta_1) \leq C_0$ or $d_S(y, \beta_1) \leq C_0$. Suppose $d_S(x, \beta_1) \leq C_0$. Then we have

\begin{equation*}
    \begin{split}
        d_{T_D} (x_0, y_0)  & = d_{T_D} (x_0,b) + d_{T_D} (b,y_0) \\
        & = d_S (x, \beta_1) + d_S (\beta_1,y) \\
        & \leq d_S (x, \beta_1) + d_S(\beta_1, x) + d_S(x,y)\\
        & \leq d_S (x, y)  + 2C_0.\\
    \end{split}
\end{equation*}

A similar computation proves the lower bound if $d_S(y, \beta_1) \leq C_0$.


\begin{figure}[h]
    \centering
    \includegraphics[trim = 11cm 4cm 6cm 7cm, clip = true, scale = 1]{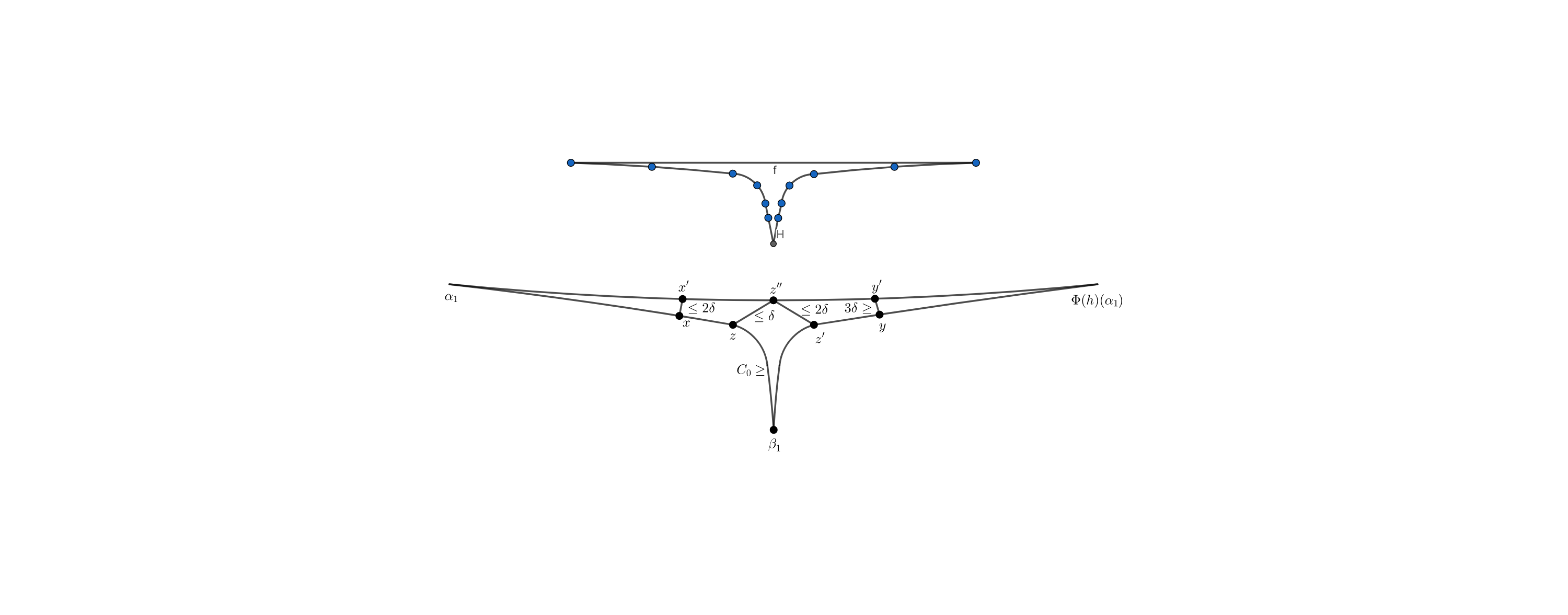}
    \caption{Lemma 3.3 Case 2}
    \label{fig:3.3 case 2}
\end{figure}

\textbf{Case 2}: $d_S(x, \beta_1), d_S(y, \beta_1) > C_0$. As in the proof of Lemma 3.2 we can find points $z \in [\alpha_1, \beta_1], z' \in [\beta_1, \Phi(h)(\alpha_1)], z'' \in [\alpha_1, \Phi(h)(\alpha_1)]$ such that $d_S(z, z'), d_S(z, z'') \leq \delta$ and $d_S(z, \beta_1), d_S(z', \beta_1) \leq C_0$. Then $x$ is on the triangle $(\alpha_1, z, z'')$ and $y$ is on the triangle $(z', z'', \Phi(h)(\alpha_1))$. Since $(\alpha_1, z, z'')$ is $\delta$-thin and $d_S(z, z'') \leq \delta$ there is a point $x' \in [\alpha_1, z'']$ such that $d_S(x, x') \leq 2\delta$. Similarly since $(z', z'', \Phi(h)(\alpha_1))$ is $\delta$-thin and $d_S(z', z'') \leq 2\delta$ there is a point $y' \in [z'', \Phi(h)(\alpha_1)]$ such that $d_S(y, y') \leq 3\delta$. Then we have

\begin{equation*}
    \begin{split}
        d_{T_D} (x_0 ,y_0)  & = d_{T_D} (x_0,\beta_1) + d_{T_D} (\beta_1,y_0) \\
        & = d_S (x,z) + d_S (z, \beta_1) + d_S (\beta_1,z') + d_S (z',y)\\
        & \leq d_S (x,z) + d_S (z',y) + 2C_0\\
        & \leq d_S (x',z'') + d_S (z'',y') + 2C_0 + 8\delta\\
        & = d_S (x',y') + 2C_0 + 8\delta\\
        & \leq d_S (x,y) + 2C_0 + 13\delta.\\
    \end{split}
\end{equation*}
\end{proof}

\begin{proof}[Proof of Theorem \ref{thm: main}]
It now follows from Th\'{e}or\`{e}me 1.4 of Chapter 3 of \cite{CDP} that there exists some $L_1, K_1, C_1$ such that for all $L>L_1$, any $L$-local $(1, 13\delta + 2C_0)$-quasigeodesic is a global $(K_1, C_1)$-quasigeodesic. Here $L_1, K_1,$ and $C_1$ depend only on $\delta, 1,$ and $13\delta + 2C_0$. Hence taking $D_0>L_1$, all geodesics in $T_D$ map to $(K_1, C_1)$-quasigeodesics in $\C$. So $\phi := \phi' \circ \phi_D:T \to \mathcal{C}(S)$ is indeed an equivariant (by construction) $(K_2, C_2)$-quasiisometric embedding. Notice that because there is a uniform bound on $\delta$, a uniform bound on $M$ (and hence on $C_0$), and $L_1$ is uniform in $\delta, 1,$ and $13\delta + 2C_0$, there is a uniform bound on $D_0$ independent of $S.$

This also proves that $H_A * H_B$ injects into $\Mod$ via $\Phi$. To see this, we must show ker$\Phi$ is trivial. Let $f \in H_A * H_B - \{1\}$. If $f$ is conjugate into a factor of $H_A * H_B$ , then $f$ maps to a multitwist. Specifically, if $f = ghg^{-1}$ with $g \in H_A * H_B$ and $h \in H_A \cup H_B$,  then $\Phi(f)$ is the multitwist about the $\Phi(g)$ image of the underlying curves of $\Phi(h) \in H_A \cup H_B \subset \Mod$, and so is nontrivial.

Suppose that $f$ is not conjugate into any factor. By conjugating if necessary, we can write $f$ as a word $g$ that starts with a syllable in $H_A$ and ends with a syllable in $H_B$. Let $v$ be the midpoint of the edge $[a, b]$ in $T$. Then the infinite geodesic path $\gamma$ in $T$ connecting $..., g^{-1}(v), v, g(v), g^2(v), ...$ maps to a quasigeodesic in $\C$ via $\phi$.

Notice that $g$ acts as translations on $\gamma$. Hence by equivariance, $\phi(\gamma)$ is a quasiaxis for the action of $\Phi(g)$ on $\C$. Then $\Phi(g)$ is pseudo-Anosov and since being pseudo-Anosov is a conjugacy invariant, so is $\Phi(f)$. Therefore ker$\Phi = \{1\}$ and $G \cong H_A * H_B$. We also have that $G$ is PGF: (1) is clear and (2) is satisfied since we've equivariantly quasiisometrically embedded $T$ into $\C$ and $T$ is equivariantly quasiisometric to $\Ghat(G, \{ H_A, H_B\})$.
\end{proof}

This method of constructing global quasigeodesics from local data is not too dissimilar from the techniques used in the proof of \cite[Theorem 1.3]{Gu}.

\section{General Marking Graphs}
 
To prove that our groups are quasiisometrically embedded in $\Mod$, we will introduce a model space $\M$ for $\Mod$ which we call a \textit{general marking graph}. Our definition is a modification of Masur and Minsky's marking graph, which we denote by $\MM,$ whose vertices are complete clean markings and with edges between markings that differ by an ``elementary move.'' See \cite[Section 2.5]{MM2} for details. We note that for our purposes, we think of complete clean markings as maximal bases (pants decompositions) together with clean transverse curves. In \cite{MM2}, Masur and Minsky define complete clean markings as maximal bases together with \textit{projections} of clean transverse curves to their respective base curves. But if you have a vertex in an annular complex which is a projection of a curve in $S,$ then there is only one curve which it is the projection of. So in either case, a complete clean marking is really the same data.

First we define $R$-markings, which will be the vertices for $\M$. An $R$-\textit{marking} is the homotopy class of a filling collection of simple closed curves $\mu$ on $S$ such that the curves in $\mu$ pairwise intersect at most $R$ times. For fixed $R$, notice there finitely many $R$-markings up to homeomorphism. That is, there are finitely many Mod$(S)$-orbits of $R$-markings. 

There exists $E \in \mathbb{N}$ such that for all $\mu, \mu' \in \mathcal{M}_0$ there exists mapping classes $g, g' \in$ Mod$(S)$ such that the curves in $g\mu, g'\mu'$ pairwise intersect at most $E$ times, i.e. $g\mu \cup g'\mu'$ is an $E$-marking.

\begin{figure}[h]
    \centering
    \includegraphics[trim = 9cm 4.5cm 6cm 3.5cm, clip = true, scale = .75]{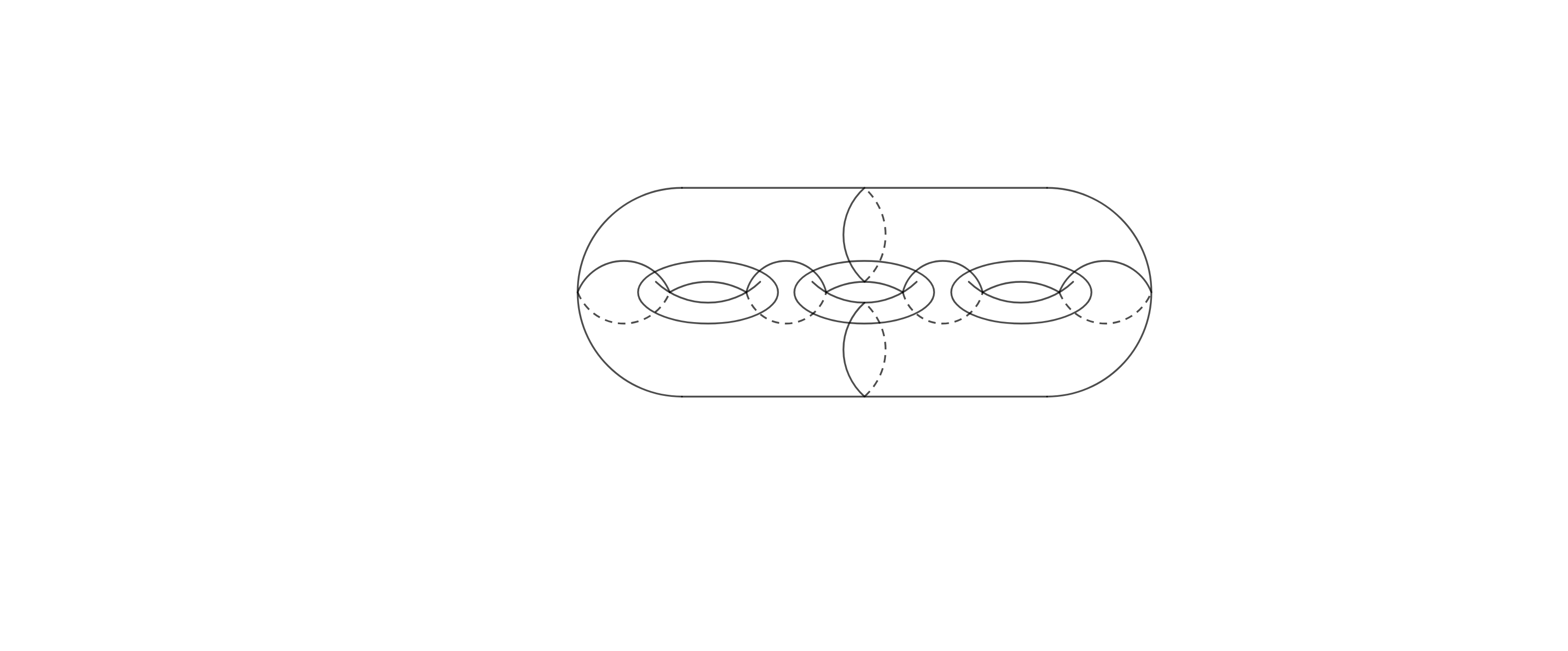}
    \caption{$\mu_0$}
    \label{fig:mu0}
\end{figure}

Consider the 1-marking $\mu_0$ above and let $h_1, ..., h_k \in$ Mod$(S)$ denote the Humphries generators. We also choose $E$ large enough so that $\mu_0 \cup h_i \mu_0$ is an $E$-marking for all $i = 1, ..., k$. We define the edges of $\M$ to be between $\mu, \mu' \in \mathcal{M}_0$ if $\mu \cup \mu'$ is an $E$-marking.

\begin{lemma}
$\M$ is connected.
\end{lemma}

\begin{proof}

Let $\mu, \mu' \in \mathcal{M}_0.$ By the above, there are mapping classes $g, g' \in \Mod$ such that we have edges $[\mu, g\mu_0], [\mu', g'\mu_0]$ in $\M$. To show connectivity it suffices to show $\mu_0$ is connected to every marking in its orbit. Let $h = h_{i_1} h_{i_2} ... h_{i_l} \in$ Mod$(S)$. Consider the edge path in $\M$ given by the vertices $\mu_0, h_{i_1}\mu_0, h_{i_1} h_{i_2} \mu_0, ...,$  $h_{i_1} h_{i_2} ... h_{i_{l-1}}\mu_0, h\mu_0$. Hence $\mu_0$ is connected to every marking in its orbit and $\M$ is connected. 
\end{proof}

$\M$ is locally finite since, given a vertex, only finitely many $R$-markings can intersect it at most $E$ times. It then follows from the above that the action of $\Mod$ on $\M$ is cocompact.

Notice that vertex stabilizers are finite. If a mapping class fixes a vertex $\mu$, it is permuting the component curves of $\mu$ and disks of $S-\mu$. Notice that these permutations uniquely determine the mapping class by the Alexander trick. Since there are finitely many permutations, the stabilizer is finite.

It follows that the action of $\Mod$ on $\M$ is properly discontinuous. To see this, let $K$ be a compact set. It is contained in a closed ball of radius say $r$ centered at some $\mu$. Suppose there exists infinitely many distinct mapping classes $g_1, g_2, ...$ such that $g_i K \cap K \neq \emptyset$. Hence for all $i$ we have $d_\M (\mu, g_i \mu) \leq 2r$. Since the ball of radius $2r$ around $\mu$ contains finitely many vertices, this means that there is a infinite sequence $i_1, i_2, ...$ such that $g_{i_1} \mu = g_{i_2} \mu = ...$. But then applying $g_{i_1}^{-1}$ gives $\mu = g_{i_1}^{-1} g_{i_2} \mu = ...$ which contradicts finiteness of vertex stabilizers. 

It then follows from Milnor-{\v S}varc that $\mathcal{M} \stackrel{\text{q.i.}}{\cong}$ Mod$(S) \stackrel{\text{q.i.}}{\cong} \mathcal{M}_{MM}.$ Next, we want to state and prove a Lipschitz result for projections of markings to any subsurface $Y$. It is a well-known fact that for $\alpha, \beta, \in \Co$ with $\xi(S) \geq 1$ we have the following bound on $d_S(\alpha, \beta)$ in terms of intersection number

\begin{equation*}
    d_S(\alpha, \beta) \leq 2\mbox{log}_2(i(\alpha, \beta)) + 2.
\end{equation*}

Define

\begin{equation*}
    k(A) = max \{2\mbox{log}_2(4(A+1)) + 2, A + 1 \}.
\end{equation*}

\begin{lemma}
Let $\mu \in \mathcal{M}_0$ and $Y \subset S$ be a domain with $\xi(Y) \neq 0$.Then \emph{diam}$_Y(\mu) \leq k(R)$. Moreover, for any edge $e$ in $\M$ we have \emph{diam}$_Y(e) \leq k(E)$.
\end{lemma}

\begin{proof}
Given any curves $\alpha, \beta \in \mu$ we have $i(\alpha, \beta) \leq R$. If $Y$ is an annulus, these lift to arcs with at most $R$ intersections, hence $\mu$ must project to a set of diameter at most $R+1$.

If $\xi(Y) \geq 1$ each intersection point becomes at most 4 intersection points when you project to $Y$. For certain $Y$, you can get 4 new intersections. See \cite[Lemma 2.2]{MM2}. Hence we have at most $4R + 4 = 4(R + 1)$ intersections in $Y$ and diam$_Y(\mu) \leq k(R)$ 
 
The moreover statement follows since edges of $\M$ are $E$-markings.
\end{proof}

Notice that by our choice of edges, $R$-markings are always adjacent to complete clean markings in $\mathcal{M}.$ That is, for any $\mu, \nu \in \mathcal{M}$ there are complete clean markings $\mu', \nu'$ such that
\begin{equation}
   d_\mathcal{M}(\mu, \mu'), d_\mathcal{M}(\nu, \nu') \leq 1. 
\end{equation}

Coupling this with our Lipschitz result, notice that for any $\mu, \nu$ and adjacent clean $\mu', \nu'$ we have
\begin{equation}
  |d_Y(\mu, \nu) - d_Y(\mu', \nu')| \leq 2k(E).
\end{equation}

We are now ready to prove a Masur-Minsky style distance formula for our general marking graph. See \cite[Theorem 6.12]{MM2}. Let $d_{Y_{MM}}$ denote distance in $Y$ of markings but via Masur and Minsky's projections $\pi_\YM$ of markings (see \cite[Section 2.5]{MM2}), which we note are different from our projections of $R$-markings, which are simply subsurface projections $\pi_Y$. Let

\begin{equation*}
   \Omega(\mu, \nu, A) = \{ Y \subset S \mbox{ } | \mbox{ } d_Y (\mu, \nu) \geq A\},
\end{equation*}
and

\begin{equation*}
    \Omega_{MM}(\mu, \nu, A) = \{ Y \subset S \mbox{ } | \mbox{ } d_\YM (\mu, \nu) \geq A\}.
\end{equation*}

\begin{thm}
Given $A_2 > 0$ sufficiently large, there exists constants $K_3(A_2) \geq 1 , C_3(A_2) \geq 0$ such that for any $\mu, \nu \in \M_0$ we have

\begin{equation*}
    d_\M(\mu, \nu) \asymp_{K_3,C_3} \sum_{Y \in S} [d_Y(\mu, \nu)]_{A_2}.
\end{equation*}
\end{thm}

\begin{proof}

This distance formula is a consequence of the following equation which we explain below. Here $\mu', \nu'$ are complete clean markings adjacent to $\mu, \nu$, respectively. 

\begin{equation*}
    \begin{split}
        d_\M(\mu, \nu)  & \asymp_{1,2}  d_\M(\mu', \nu')\\
        & \asymp d_{\M_{MM}}(\mu', \nu')\\
        & \asymp_{K_0(A_0),C_0(A_0)} \sum_{Y \in \Omega_{MM}(\mu', \nu', A_0)} d_{Y_{MM}}(\mu', \nu')\\
        & \asymp_{ 2 ,0} \sum_{Y \in \Omega_{MM}(\mu', \nu', A_0)} d_Y(\mu', \nu')\\
        & \asymp_{K_1(A_1), C_1(A_1)} \sum_{Y \in \Omega(\mu', \nu', A_1)} d_Y(\mu', \nu')\\
        & \asymp_{2,0} \sum_{Y \in \Omega(\mu', \nu', A_1)} d_Y(\mu, \nu)\\
        & \asymp_{K_2(A_2), C_2(A_2)} \sum_{Y \in \Omega(\mu, \nu, A_2)} d_Y(\mu, \nu).\\
    \end{split}
\end{equation*}

 The first comparability is by (4). The second comparability is by Milnor-{\v S}varc. The third is Masur-Minsky distance formula. The last 4 require justification.

If $Y$ is an annulus with core a base curve $\alpha$ of $\mu'$ with transversal $\beta$ then $\pi_Y(\mu') = \pi_\YM (\mu') = \pi_Y(\beta)$ by ``cleanliness.'' If instead $Y$ is an annulus with the core a transversal $\beta$ transverse to $\alpha$, then $\pi_{\YM} (\mu')= \pi_Y ($base$(\mu')) = \pi_Y (\alpha)$. The second equality is because $\alpha$ is the only base curve $\beta$ intersects. But for us the projection $\pi_Y(\mu')$ will include any other curves $\beta$ intersects. By cleanliness this is at most 4 other transversals. But these transversals all lie in the link of $\alpha$ (by cleanliness) hence by Theorem 2.3, they project to a set of diameter at most 2 in $\mathcal{C}(Y)$ with $\alpha$ in the middle of this set. Hence in $\mathcal{C}(Y)$ we have $|d_Y (\mu', \nu') - d_\YM (\mu', \nu')| \leq 2 + 2 = 4$ for any other complete clean $\nu'.$

Suppose $Y$ is annulus with core not a base curve or transversal of $\mu'$ or $\xi(Y) > 0$. Again the difference between the projections in \cite{MM2} and ours is that $\pi_\YM(\mu') = \pi_Y ($base$(\mu'))$ and we are projecting all of $\mu'$. In $\C$, base$(\mu')$ is a simplex and $\mu'$ is a simplex with adjacent edges and has diameter 3. By Theorem 2.3, $\pi_\YM(\mu')$ is a set of diameter at most 2 exactly in the middle of $\pi_Y (\mu')$ which has diameter at most 6. It follows that $|d_Y (\mu', \nu') - d_\YM (\mu', \nu')| \leq 2 + 2 = 4$ for any other complete clean marking $\nu'.$

It follows that for $Y$ with $\xi(Y) \neq 0$, if $d_\YM (\mu', \nu') \geq A_0$ (i.e. $Y \in \Omega_{MM}(\mu', \nu', A_0)$), then

\begin{equation*}
    \frac{A_0 -4}{A_0} \leq \frac{d_Y (\mu', \nu')}{d_\YM (\mu', \nu')} \leq \frac{A_0 +4}{A_0},
\end{equation*}
and it follows that

\begin{equation*}
    \frac{A_0 -4}{A_0} \leq \frac{ \sum_{Y \in \Omega_{MM}(\mu', \nu', A_0)} d_Y (\mu', \nu')}{\sum_{Y \in \Omega_{MM}(\mu', \nu', A_0)} d_\YM (\mu', \nu')} \leq \frac{A_0 + 4}{A_0}.
\end{equation*}

Choosing $A_0>8$ gives (2,0)-comparability. For the fifth comparability, letting $A_1>A_0 + 4$ gives

\begin{equation*}
    \Omega_{MM}(\mu', \nu', A_1) \subset \Omega (\mu', \nu', A_1) \subset \Omega_{MM}(\mu', \nu', A_0),
\end{equation*}
where the first inclusion is obvious. Since summing $d_Y(\mu', \nu')$ over the first and last set is comparable to $d_{\M_{MM}}(\mu', \nu')$, so is summing over the middle set, for some constants $K_1(A_1), C_1(A_1)$. For the sixth comparability, equation (5) gives

\begin{equation*}
    \frac{A_1 - 2k(E)}{A_1} \leq \frac{d_Y (\mu, \nu)}{d_Y (\mu', \nu')} \leq \frac{A_1 + 2k(E)}{A_1}.
\end{equation*}

Choosing $A_1 > 4k(E)$ gives $(2,0)$-comparability. For the seventh comparability notice that

\begin{equation*}
    \Omega(\mu', \nu', A_1) \subset \Omega (\mu, \nu, A_1 - 2k(E)) \subset \Omega(\mu', \nu', A_1 - 4k(E)).
\end{equation*}

For $A_1$ sufficiently large, the sums of $d_Y(\mu, \nu)$ over the first and last set are comparable, hence so is summing over the middle set. Letting $A_2 = A_1 - 2k(E)$ finishes the proof.
\end{proof}

\section{Quasiisometrically Embedding into \Mod}

Convex cocompact subgroups are quasiisometrically embedded in $\Mod$. This follows from the work of Hamenst{\"a}dt \cite{H} and Kent-Leininger \cite{KL}, together with the Masur-Minsky distance formula:

\begin{equation*}
   d_{\Mod} (g,h) \leq d_G (g,h) \asymp d_S (gv,hv) \preceq_{1,\kappa} \sum_{Y \subset S} [d_Y (gv,hv)]_\kappa \asymp d_{\Mod} (g,h).
\end{equation*}

The first inequality follows from the triangle inequality. The next comparability is by \cite{H, KL}, where $v$ is some fixed vertex in $\C$. The inequality after that is obvious for say $\kappa > 0.$ And the last comparability is the Masur-Minsky distance formula with $\kappa$ sufficiently large. Finitely generated Veech groups are also quasiisometrically embedded in $\Mod$ by work of Tang \cite{T}. It turns out that our groups are too. Notice that in the statement below, $A$ and $B$ needn't be ``sufficiently far apart'' in $\C$. That is, it holds for any PGF group $G = H_A*H_B,$ hence applies to more general examples than the groups described in Theorem 1.1.

\begin{uundistorted}
Let $G = H_A * H_B \subset$ \emph{Mod}$(S)$ be a nontrivial free product which is PGF with $H_A, H_B$ subgroups generated by multitwists about multicurves $A, B,$ respectively. Then $G$ is undistorted in \emph{Mod}$(S)$.
\end{uundistorted}

A key ingredient in the convex cocompact and Veech group cases is the Masur-Minsky distance formula. The approach to proving the formula is employing a model space for $\Mod$, what we earlier called $\MM$. We prove Theorem 1.2 by quasiisometrically embedding a model space for $G$ into the general marking graph $\M$ which is a model for $\Mod$.

We use the same notation from section 3. Let $V = \{ p^{-1}(x)\}$ denote the preimage of $x$ in $\widetilde{X}$. Notice that the collapsing of $\widetilde{X}$ to $T$ is injective  on $V$, so we also think of $V$ as a subset of $T$. Fixing some $v_0 \in V$, the orbit map $G \to \widetilde{X}$ given by $g \mapsto gv_0$ is a quasiisometry. Define a map $v:\{p^{-1}(t_A)\} \cup \{p^{-1}(t_B)\} \to V$ that sends points to the closest point of $V$.

The components of the preimages of $T_A$, $T_B$ are copies of $\R^n, \R^m$, respectively. We call them \textit{flats}. Recall that each flat corresponds to a coset $gH_A, gH_B.,$ and that $W$ is the set of all vertices in $T$ that are obtained by collapsing the flats to points $ga \in W_A$, $gb \in W_B$. Given $w \in W,$ let $F(w)$ denote the flat associated to $w$ and let $A(w)$ denote the multicurve associated to $w$.

Define $\pi_w: \widetilde{X} \to F(w)$ to be a closest point projection. Letting $d_{\widetilde{X}}$ denote the metric on $\widetilde{X}$, we write $d_w = d_{\widetilde{X}}|_{F(w)}$, and for $x_1, x_2 \in \widetilde{X}$ we write

\begin{equation*}
    d_w(x_1, x_2) = d_w(\pi_w(x_1), \pi_w(x_2)).
\end{equation*}

Let $\pi_T: \widetilde{X} \to T$ denote the collapsing map described above. Letting $d_T$ denote the metric on $T$, for $x_1, x_2 \in \widetilde{X}$ we write

\begin{equation*}
    d_T(x_1, x_2) = d_T(\pi_T(x_1), \pi_T(x_2)).
\end{equation*}

We have the following distance formula in $\widetilde{X}$.

\begin{lemma}
For all $x_1, x_2 \in \widetilde{X}$ we have

\begin{equation*}
   d_{\widetilde{X}} (x_1, x_2) = d_T(x_1, x_2) + \sum_{w \in W} d_w(x_1, x_2).
\end{equation*}

\end{lemma}

\begin{proof}

Let $\gamma$ denote the geodesic in $\widetilde{X}$ between $x_1$ and $x_2$. Then $\gamma$ decomposes into geodesic segments in the flats $F(w)$ and the edges connecting the flats, specifically the components of the preimage of the edge $[0,1]$ in $X$. The segment of the geodesic in $F(w)$ starts at $\pi_w(x_1)$ and ends at $\pi_w(x_2)$ hence has length exactly $d_w(x_1, x_2)$. The contribution from traveling across $p^{-1}([0,1])$ is exactly $d_T(x_1, x_2)$.
\end{proof}

Since $\Ghat$ is equivariantly quasiisometric to $T,$ parabolic geometric finiteness provides a $G$-equivariant $(K_0, C_0)$-quasiisometric embedding $\phi: T \to \C.$ Hence, any choice of a coarsely Lipschitz $G$-equivariant map from $T$ to $\C$ will be a quasiisometric embedding. So, as in the setup to the proof of Theorem 1.1, we define $\phi(ga) = g \alpha, \phi(gb) = g \beta,$ and $\phi(g[a,b]) = g[\alpha, \beta]$ with $\alpha, \beta$ some components of $A,B,$ respectively and $[\alpha, \beta]$ some choice of a geodesic segment in $\C.$ Notice for any $w \in W,$ we have that $\phi(w)$ is a vertex of the simplex $A(w)$ in $\C.$

We now fix an $R$-marking $\mu$ that ``lies in the middle'' of $[\alpha, \beta].$ Specifically, $\mu$ contains $\phi(\tilde{x})$ if $\phi(\tilde{x}) \in \Co.$ If $\phi(\tilde{x}) \notin \Co,$ we perturb $x \in [0,1]$ so that $\phi(\tilde{x}) \in \Co$ and again choose $\mu$ so that it contains $\phi(\tilde{x}).$ We define a equivariant map $\mu: V \to \M$ given by $\mu(gv) = g\mu$. This gives us a coarse map $\mu: \widetilde{X} \to \M$. To prove Theorem 1.2, it suffices to show that there exists $K \geq 1, C \geq 0$ such that for any given $v_1, v_2 \in V$ we have

\begin{equation*}
    \frac{1}{K}d_\M(\mu(v_1), \mu(v_2)) - C \leq d_{\widetilde{X}} (v_1, v_2) \leq Kd_\M(\mu(v_1), \mu(v_2)) + C.
\end{equation*}

We will need the following series of lemmas.

\begin{lemma}
Let $\kappa > 0$. Then for all $v_1, v_2 \in V$ we have

\begin{equation*}
    [d_T(v_1, v_2)]_\kappa + \sum_{w \in W} [d_{w}(v_1, v_2)]_\kappa  \asymp_{\kappa+1, \kappa^2 + 2\kappa} d_T(v_1, v_2) + \sum_{w \in W} d_{w}(v_1, v_2).
\end{equation*}
    
\end{lemma}

\begin{proof}
The comparability lower bound is obvious, so we show the upper bound. Define

\begin{equation*}
    \Omega_\geq (\kappa, v_1, v_2) = \{w \in W | d_w (v_1, v_2) \geq \kappa\},
\end{equation*}

\begin{equation*}
    \Omega_< (\kappa, v_1, v_2) = \{w \in W | 0 < d_w (v_1, v_2) < \kappa\}.
\end{equation*}

If $x = \frac{1}{2}$, notice that 

\begin{equation*}
    d_T (v_1, v_2) = |\Omega_\geq| + |\Omega_<|. 
\end{equation*}

If $x$ was perturbed so that $\phi(\tilde{x}) \in \Co$, we have 

\begin{equation*}
    d_T (v_1, v_2) + 1 \geq |\Omega_\geq| + |\Omega_<|. 
\end{equation*}

We have

\begin{equation*}
    \begin{split}
         d_T(v_1, v_2) + \sum_{w \in W} d_w (v_1, v_2) & = d_T(v_1, v_2) + \sum_{w \in \Omega_<} d_w (v_1, v_2) +\sum_{w \in \Omega_\geq} d_w (v_1, v_2) \\ 
         & \leq d_T(v_1, v_2) +  \kappa|\Omega_<| +\sum_{w \in \Omega_\geq} d_w (v_1, v_2) \\
         & \leq d_T(v_1, v_2) +  \kappa (d_T(v_1, v_2) + 1) + \sum_{w \in \Omega_\geq} d_w (v_1, v_2)  \\
         & = (\kappa + 1)d_T(v_1, v_2) + \sum_{w \in W} [d_w (v_1, v_2)]_\kappa + \kappa\\
         & \leq (\kappa + 1)[d_T(v_1, v_2)]_\kappa + \sum_{w \in W} [d_w (v_1, v_2)]_\kappa + \kappa + (\kappa + 1)\kappa\\
         & < (\kappa + 1)([d_T(v_1, v_2)]_\kappa + \sum_{w \in W} [d_w (v_1, v_2)]_\kappa) + \kappa + (\kappa + 1)\kappa. \\
    \end{split}
\end{equation*}

The second to last line follows because
\begin{equation*}
    d_T(v_1, v_2) < \kappa \Rightarrow (\kappa + 1)d_T(v_1, v_2) < (\kappa+1)\kappa.
\end{equation*}
\end{proof}

\begin{lemma}
Let $G$ be PGF with respect to $H_1,..., H_k$ which are subgroups generated by multitwists about multicurves $A_1, ..., A_k$, respectively. Let $\mathcal{A} = \bigcup_{g \in G} g(A_i)$, i.e. $\mathcal{A}$ is the union of all the $G$-orbits of the multicurves. Then for any distinct $A, A' \in \mathcal{A}$, $A \cup A'$ fills $S$.
\end{lemma}

\begin{proof}
Suppose not. Let $H, H'$ denote the cosets corresponding to $A, A'$. Since $A \cup A'$ doesn't fill $S,$ there are some multitwists $t \in H, t' \in H'$ and a curve $\gamma \in \Co$ contained in $S- A \cup A'$ such that $t(\gamma) = t'(\gamma) = t^kt'^k (\gamma) = \gamma$ for all $k \in \Z$. If $[t, t'] = 1$ then $t$ and $t'$ span a $\Z^2$ plane in $\G,$ the Cayley graph of $G,$ and the $1$-neighborhoods in $\G$ of $H$ and $t'H$ intersect in an infinite set which contradicts the BCP property. Hence $[t,t'] \neq 1$ and the underlying curves or multicurves of $t,t'$, say $\alpha, \alpha'$ intersect. It is a well-known fact (for example see \cite{Th}) that there exists some $N$ such that $t^Nt'^N$ is pseudo-Anosov on the subsurface filled by $\alpha \cup \alpha'$. Hence $t^Nt'^N$ is not conjugate into $H_1, ..., H_k$ and has an axis in $\Ghat$. Since $G$ is PGF, there is a corresponding quasiaxis in $\C$ and the $t^Nt'^N$ acts with positive translation length on $\C$. This contradicts $t^Nt'^N(\gamma) = \gamma$.
\end{proof}

We recall the following well-known theorem and state it without proof. For a reference, see \cite[Theorem III.H.1.7]{BH}.

\begin{thm}[Stability of quasigeodesics]
For all $K \geq 1, C \geq 0, \delta > 0$ there exists $R(K, C, \delta)$ with the following property. If $X$ is a $\delta$-hyperbolic metric space, $\gamma$ is a $(K, C)$-quasigeodesic segment in $X$ and $\gamma'$ is a geodesic segment between the endpoints of $\gamma$, then the Hausdorff distance between $\gamma$ and $\gamma'$ is at most $R$.
\end{thm}

The following lemma is likely a well-known result.

\begin{lemma}
Let $X$ be a $\delta$-hyperbolic space, $\gamma: [0,L] \to X$ a $(K,C)$-quasigeodesic, $R(K,C,\delta)$ the stability constant, and $\pi: \gamma([0,L]) \to [\gamma(0), \gamma(L)]$ a closest point projection. If $s,t \in [0,L]$ such that $s<t$ and $|t-s|>P := 2K(C+2R)+K^2$, then $\pi(\gamma(s)) < \pi(\gamma(t))$.
\end{lemma}

\begin{proof}
Using the quasigeodesic inequalities and stability constants gives

\begin{equation*}
    \frac{1}{K}|t-s| - C - 2R  \leq d_X(\pi(\gamma(t)), \pi(\gamma(s))) \leq K|t-s| + C + 2R.
\end{equation*}

Notice that in particular for all $t_0 \in [0,L]$

\begin{equation*}
    d_X(\pi(\gamma(t_0 +1)), \pi(\gamma(t_0))) \leq K + C + 2R.
\end{equation*}

Suppose $\pi(\gamma(s)) \geq \pi(\gamma(t))$. Since the geodesic $[\gamma(0), \gamma(L)]$ must eventually travel from $\pi(\gamma(t))$ to $\gamma(L)$ in bounded increments of length at most $K + C + 2R$, there exists $t' \geq t$ such that

\begin{equation*}
    d_X(\pi(\gamma(t')), \pi(\gamma(s))) < K + C + 2R.
\end{equation*}

But then using the lower bound from the quasigeodesic inequality and stability constants we get

\begin{equation*}
    d_X(\pi(\gamma(t')), \pi(\gamma(s))) \geq \frac{1}{K}|t'-s| - C - 2R \geq \frac{1}{K}|t-s| - C - 2R > K + C + 2R,
\end{equation*}a contradiction. Hence $\pi(s) < \pi(t)$.
 
\end{proof}

Let $R_0(K_0, C_0, \delta)$ be a stability constant for the image of $\phi$ in $\C$. Since we may take $R_0$ as large as we like, assume it is an integer greater than 1. For $t_1, t_2 \in T$ we write $[t_1, t_2]_T$ to denote the geodesic between them in $T$.

\begin{lemma}
Let $v_1, v_2 \in V$, $t \in T$, $w \in W$, such that $t, w\in [v_1, v_2]_T$, $t < w, d_T(t,w) > P_0 := 2K_0(C_0 + 2R_0) + K_0^2,$ and $\pi:\phi([v_1, v_2]_T) \to [\phi(v_1), \phi(v_2)]$ is a closest point projection map. If $s \in [\phi(v_1), \pi(\phi(t))]$ then

\begin{equation*}
    \pi_{\phi(w)}(s) \neq \emptyset.
\end{equation*}

Moreover for any simplex $\Delta$ containing $\phi(w)$ we have

\begin{equation*}
    \pi_{\Delta}(s) \neq \emptyset.
\end{equation*}

\end{lemma}

Recall that $\pi_\Delta$ is the product of the projections of its vertices and hence $\pi_\Delta(s) \neq \emptyset$ means that for each vertex $p$ of $\Delta$ we have $\pi_p(s) \neq \emptyset.$

\begin{figure}[t]
    \centering
    \includegraphics[trim = 16cm 7.5cm 8cm 7cm, clip = true, scale = 1]{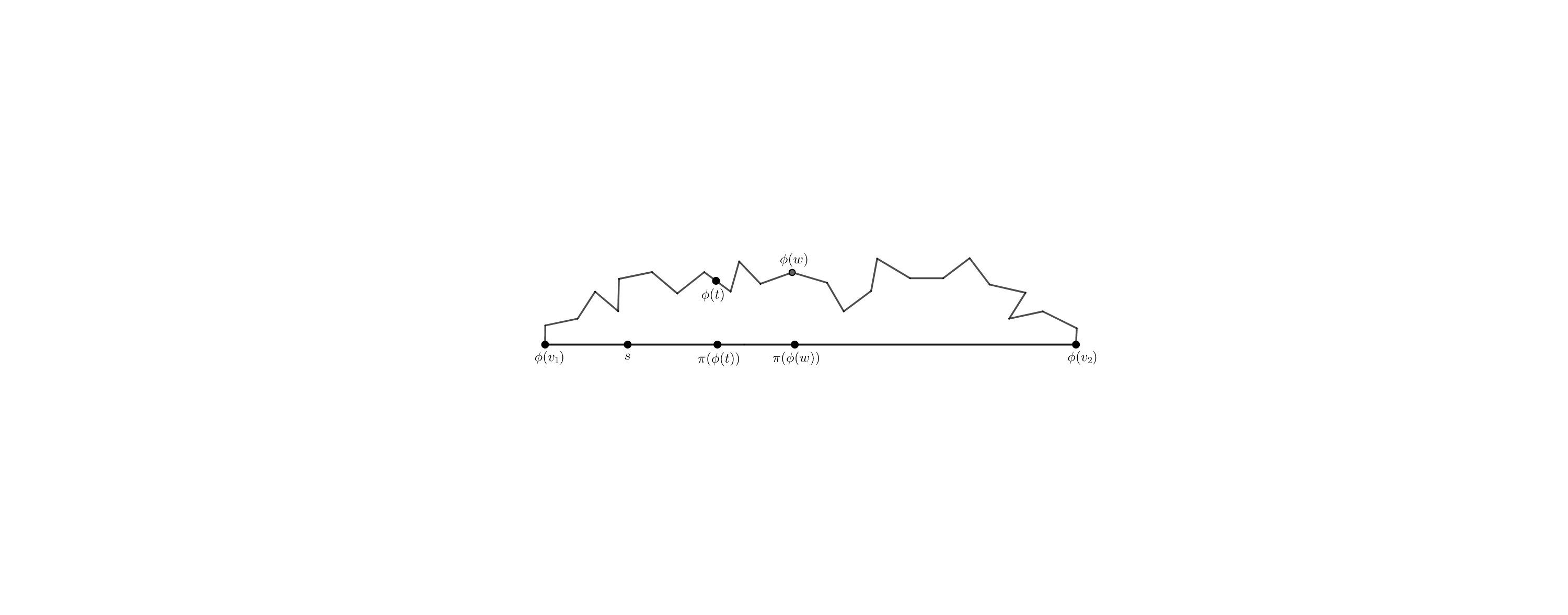}
    \caption{Lemma 5.6}
    \label{fig:my_label}
\end{figure}

\begin{proof}
Under the assumptions of the hypothesis, it follows from Lemma 5.5 that

\begin{equation*}
    \pi(\phi(t)) < \pi(\phi(w)).
\end{equation*}

Notice

\begin{equation*}
    d_S(\pi(\phi(t)), \pi(\phi(w))) > \frac{P_0}{K_0} - C_0 - 2R_0 = 2R_0 + K_0 + C_0 \geq 2R_0 + 1.
\end{equation*}

The last inequality is because at worst, $K_0=1, C_0=0$. Since $s \in [\phi(v_1), \pi(\phi(t))]$ we have
\begin{equation*}
    d_S(s, \pi(\phi(w)))  > 2R_0+1 \Rightarrow d_S(s, \phi(w))  > R_0+1 \Rightarrow \pi_{\phi(w)}(s) \neq \emptyset.
\end{equation*}

The moreover statement follows since the above implies for any vertex $p\in \Delta$
\begin{equation*}
    d_S(s, p)  > R_0 \Rightarrow \pi_p (s) \neq \emptyset.
\end{equation*}

\end{proof}

\begin{lemma}
There exists a constant $C_1 \geq 0$ with the following property. For all edges $e = [w_1, w_2] \subset T$ and any $Z$ a union of disjoint domains $Z_1, ..., Z_N \subset S$ with $\pi_{Z_i}(A(w_1)), \pi_{Z_i}(A(w_2)) \neq \emptyset$ for $i = 1, ..., N$ we have

\begin{equation*}
    d_{Z}(A(w_1), A(w_2)) \leq C_1.
\end{equation*}

In particular, for any $w \in  W$ with $w \neq w_1, w_2$

\begin{equation*}
    d_{A(w)}(A(w_1), A(w_2)) \leq C_1.
\end{equation*}
\end{lemma}

\begin{proof}
Fix an edge $e = [w_1, w_2]$. The number of intersections between $A(w_1)$ and $A(w_2)$ bounds the distance between their projections to any subsurfaces to which they have nonempty projections. Then there is a constant $C_1'(w_1, w_2)$ such that for all $Y \subset S$ with $\pi_Y(A(w_1)), \pi_Y(A(w_2)) \neq \emptyset$

\begin{equation*}
    d_Y (A(w_1), A(w_2)) < C_1'.
\end{equation*}

Notice that $C_1'$ does not depend on $Y$. Since $Z$ has at most $\xi(S) = 3g+p-3$ components there exists $C_1(S, w_1, w_2)$ such that

\begin{equation*}
    d_Z (A(w_1), A(w_2)) < C_1.
\end{equation*}

Again, notice $C_1$ does not depend on $Z$. Let $e' = [w_3, w_4]$ be another edge with $\pi_{Z_i}(A(w_3)), \pi_{Z_i}(A(w_4)) \neq \emptyset$ for $i = 1, ..., N$. Recall that there is a single edge orbit in $T$. So there is some $g \in G$ such that $ge=e'$. Since for all $g' \in G, w' \in W$ we have $A(g'w') = g'A(w')$ it follows that

\begin{multline*}
    d_Z (A(w_3), A(w_4))  = d_Z (A(gw_1), A(gw_2)) \\ = d_Z (gA(w_1), gA(w_2)) = d_{g^{-1}Z} (A(w_1), A(w_2)) \leq C_1.
\end{multline*}

For the ``in particular'' statement notice that for any $w, w' \in W$, $w \neq w'$ implies $\pi_{A(w)} (A(w')) \neq \emptyset$ since $G$ is PGF and hence $A(w) \cup A(w')$ is a marking by Lemma 5.3.
\end{proof}

Define a \textit{half-edge} to be a segment of length $\frac{1}{2}$ in $T$ with one endpoint in $W$ and the other in $V$.

\begin{lemma}

There exists a constant $C_2 \geq 0$ with the following property. For all half-edges $h = [v, w] \subset T$ and any Z a union of disjoint domains $Z_1, ..., Z_N \subset S$ with $\pi_{Z_i}(A(w)) \neq \emptyset$
for $i = 1, ..., N$ we have

\begin{equation*}
    d_Z(\mu(v), A(w)) \leq C_2.
\end{equation*}

In particular, for any $w' \in W$ with $w' \neq w$

\begin{equation*}
    d_{A(w')} (\mu(v), A(w)) \leq C_2.
\end{equation*}
\end{lemma}

\begin{proof}

The proof is similar to the proof of the preceding lemma. But unlike the preceding lemma, there are two half-edge orbits. So first assume that $w \in W_A$, run through the argument, and obtain a bound. Using the same argument, find a bound when $w \in W_B$ and take $C_2$ to be the maximum of the two cases.
\end{proof}

\begin{lemma}
Let $v_1, v_2 \in V$. Then for all $w \in W$ such that $d_w(v_1, v_2)>0$ we have for $i=1,2$
\begin{equation*}
    d_{A(w)}(\mu(v_i), \mu(v(\pi_w(v_i)))) \leq C_3,
\end{equation*}
where $C_3 = P_0 C_1 + 2C_2 + \xi(S)(2M + k(R) + 2)$ and $M$ is the constant from Theorem 2.2.

\end{lemma}

\begin{figure}[h]
    \centering
    \includegraphics[trim = 19cm 9.5cm 3cm 9cm, clip = true, scale = .8]{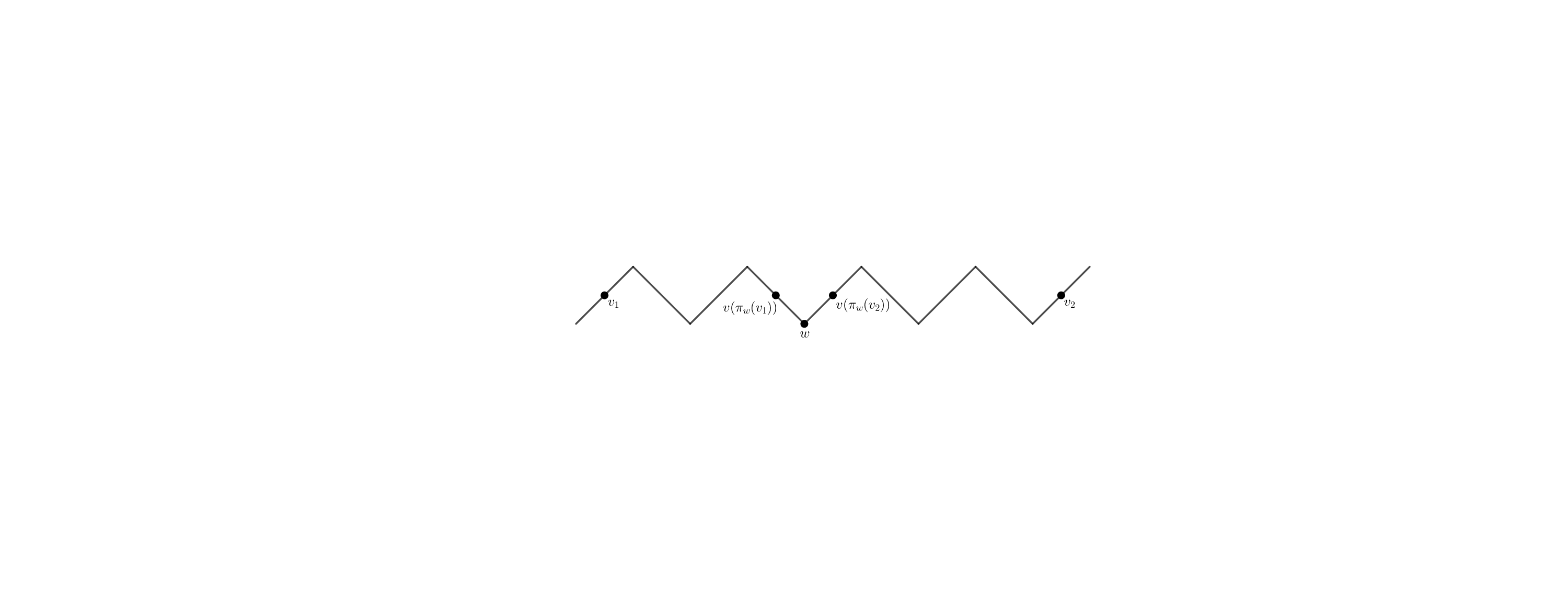}
    \caption{Lemma 5.9}
    \label{fig:5.9}
\end{figure}

\begin{proof}
There are 4 cases describing whether the points $v_1, v_2, w$ are $P_0$-close in $T$. The worst case is when $d_T(v_i, w) > P_0$ for $i=1,2$ in the sense that the bound we derive from it is the largest and works for the other 3 cases. So assume $d_T(v_i, w) > P_0.$ Notice that by Lemma 5.6 this implies $\pi_{A(w)}(\phi(v_i)) \neq \emptyset.$

Consider the geodesic $[v_1, v_2]_T$ in $T$. Label vertices of $W$ traversed in order from $v_1$ to $v_2$ by $w_1, ..., w_N$. That is
\begin{equation*}
   [v_1, v_2]_T = [v_1, w_1]_T \cup [w_1, w_2]_T \cup ... \cup [w_N, v_2]_T.
\end{equation*}
Suppose $w_i = w$. Let $1\leq j \leq N$ be the smallest number such that $d_T(w_j, w) \leq P_0$ . Either $j=1$ or $d_T(w_{j-1}, w) > P_0$.

\textbf{Case 1:} $j = 1.$

We have by the triangle inequality and Lemmas 5.7, 5.8

\begin{multline*}
    d_{A(w)} (\mu(v_1), \mu(v(\pi_w(v_1)))) \leq d_{A(w)} (\mu(v_1), A(w_1)) + \sum_{l=1}^{i-2} d_{A(w)} (A(w_l), A(w_{l+1})) \\ + d_{A(w)} (A(w_{i-1}), \mu(v(\pi_w(v_1)))) \leq (P_0 -1)C_1 + 2C_2.
\end{multline*}

\textbf{Case 2:} $d_T(w_{j-1}, w) > P_0.$

It follows from Lemma 5.6 that for all $s \in [\phi(v_1), \pi(\phi(w_{j-1}))]$ we have

\begin{equation*}
    \pi_{A(w)} (s) \neq \emptyset,
\end{equation*}
where $\pi: \phi([v_1, v_2]_T) \to [\phi(v_1), \phi(v_2)]$ is a closest point projection map. Hence it follows from Theorem 2.2, Theorem 2.3, and Lemma 4.2 that

\begin{equation*}
    d_{A(w)}(\mu(v_1), \pi(A(w_{j-1}))) \leq \xi(S)(M + 2 + k(R)).
\end{equation*}

The $\phi$-quasiisometry lower bound gives

\begin{equation*}
    d_S(\phi(w_{j-1}), \phi(w)) \geq \frac{1}{K_0}d_T(w_{j-1}, w) - C_0 > \frac{1}{K_0}P_0 - C_0 = 4R_0 + K_0 + C_0 \geq 4R_0 + 1.
\end{equation*}

Since $d_S(\phi(w_{j-1}), \pi(\phi(w_{j-1}))) \leq R_0$ we have for all $s \in [\phi(w_{j-1}), \pi(\phi(w_{j-1}))]$

\begin{equation*}
    \pi_{A(w)}(s) \neq \emptyset,
\end{equation*}
and hence again by Theorem 2.2, Theorem 2.3, and Lemma 4.2 we have that

\begin{equation*}
    d_{A(w)}(\mu(v_1), A(w_{j-1})) \leq \xi(S)(2M + 2 + k(R)).
\end{equation*}

By the triangle inequality we have

\begin{multline*}
    d_{A(w)} (A(w_{j-1}), \mu(v(\pi_{w}(v_1)))) \leq \sum_{l=j-1}^{i-2} d_{A(w)} (A(w_l), A(w_{l+1}))  + d_{A(w)} (A(w_{i-1}), \mu(v(\pi_{w}(v_1)))) \\ \leq P_0 C_1 + C_2.
\end{multline*}

Combining this with the above we have

\begin{equation*}
    d_{A(w)}(\mu(v_1), \mu(v(\pi_w(v_1)))) \leq P_0 C_1 + C_2 + \xi(S)(2M + k(R) + 2).
\end{equation*}

So in any case

\begin{equation*}
    d_{A(w)}(\mu(v_1), \mu(v(\pi_w(v_1)))) \leq P_0 C_1 + 2C_2 + \xi(S)(2M + k(R) + 2).
\end{equation*}

The same argument works for $v_2$
\end{proof}

Given any $w \in W_A,$ let $g(w)$ denote an element such that associated coset of $w$ is $g(w)H_A.$ Let $H(w) = g(w) H_A g(w)^{-1}$ and let $H_{\mbox{max}}(w)$ denote the group generated by Dehn twists about each component of $A(w).$ $H_{\mbox{max}}(w)$ is maximal in the sense that, whatever the generators of $H_A$ may be, it contains $H(w)$ as a subgroup. If $S_A$ is a generating set for $H_A$,  let $|| \cdot ||$ denote the minimal word length in $H(w)$ with its generating set $g(w) S_A g(w)^{-1}.$ Given $v_1, v_2 \in V$ let $h(v_1), h(v_2)$ denote the elements of $H(w)$ that correspond to $\pi_w(v_1), \pi_w(v_2),$ respectively. Notice that since Euclidean metrics are quasiisometric to their respective $L^1$-metrics, there exists constants so that such that

\begin{equation*}
    d_w(v_1, v_2) \asymp ||h(v_1)^{-1}h(v_2)||.
\end{equation*}

Let $|| \cdot ||_{\mbox{max}}$ denote the minimal word length in $H_{\mbox{max}}(w)$ with generating set the Dehn twists about $A(w)$ and let $h_{\mbox{max}}(v_1), h_{\mbox{max}}(v_2)$ denote the elements of $H_{\mbox{max}}(w)$ that correspond to $\pi_w(v_1), \pi_w(v_2),$ respectively. It is easy to see that

\begin{equation*}
    ||h_{\mbox{max}}(v_1)^{-1}h_{\mbox{max}}(v_2)||_{\mbox{max}} \asymp d_{A(w)}(\phi(v(\pi_w(v_1))), \phi(v(\pi_w(v_2)))),
\end{equation*}
for some constants. Since the inclusion of $H(w)$ into $H_{\mbox{max}}(w)$ is a quasiisometric embedding, it then follows that

\begin{equation*}
    d_w(v_1, v_2) \asymp_{K_A, C_A} d_{A(w)}(\phi(v(\pi_w(v_1))), \phi(v(\pi_w(v_2)))),
\end{equation*}
for some constants $K_A \geq 1, C_A \geq 0$. Then by Lemma 2.1, assuming $\kappa_A > 2K_A C_A,$ we have

\begin{equation*}
    \sum_{w \in W_A} [d_w(v_1, v_2)]_{\kappa_A} \preceq_{2K_A, 0} \sum_{w \in W_A} [d_{A(w)}(\phi(v(\pi_w(v_1))), \phi(v(\pi_w(v_2))))]_{C_A},
\end{equation*}
and since $\mu(v(\pi_w(v_i)))$ contains $\phi(v(\pi_w(v_i)))$ we have

\begin{equation*}
    \sum_{w \in W_A} [d_w(v_1, v_2)]_{\kappa_A} \preceq_{2K_A, 0} \sum_{w \in W_A} [d_{A(w)}(\mu(v(\pi_w(v_1))), \mu(v(\pi_w(v_2))))]_{C_A}.
\end{equation*}

A similar inequality holds for $w \in W_B$ with constants say $K_B \geq 1, C_B \geq 1, \kappa_B > K_B C_B$. Taking $\kappa_4 = max\{ \kappa_A, \kappa_B\}, K_4 = max\{K_A, K_B\}, C_4 = min\{C_A, C_B\}$ yields

\begin{equation}
    \sum_{w \in W} [d_w(v_1, v_2)]_{\kappa_4} \preceq_{2K_4, 0} \sum_{w \in W} [d_{A(w)}(\mu(v(\pi_w(v_1))), \mu(v(\pi_w(v_2))))]_{C_4}.
\end{equation}

We are now ready to prove Theorem 1.2.

\begin{proof}[Proof of Theorem \ref{thm: undistorted}]

Recall that it suffices to show that there exists $K \geq 1, C \geq 0$ such that for any given $v_1, v_2 \in V$ we have

\begin{equation*}
    \frac{1}{K}d_\M(\mu(v_1), \mu(v_2)) - C \leq d_{\widetilde{X}} (v_1, v_2) \leq Kd_\M(\mu(v_1), \mu(v_2)) + C.
\end{equation*}

The coarse lower bound follows from the triangle inequality so we prove the upper bound.

\begin{equation*}
    \begin{split}
        d_{\widetilde{X}}(v_1, v_2)  & = d_T(v_1, v_2) + \sum_{w \in W} d_{w}(v_1, v_2)\\
        & \preceq_{\kappa_0+ 1, \kappa_0^2 + 2\kappa_0} [d_T(v_1, v_2)]_{\kappa_0} + \sum_{w \in W} [d_w(v_1, v_2)]_{\kappa_0} \\
        & \preceq_{2K_0, 0} [d_S(\phi(v_1), \phi(v_2))]_{C_0} + \sum_{w \in W} [d_w(v_1, v_2)]_{\kappa_0} \\
        &\preceq_{2K_4, 0} [d_S(\phi(v_1), \phi(v_2))]_{C_0} + \sum_{w \in W} [d_{A(w)}(\mu(v(\pi_w(v_1))), \mu(v(\pi_w(v_2))))]_{C_4} \\
        & \preceq_{2, 0} [d_S(\phi(v_1), \phi(v_2))]_{C_0} + \sum_{w \in W} [d_{A(w)}(\mu(v_1), \mu(v_2))]_{2C_3}\\
        & \asymp_{\xi(S) + 1, 0} \sum_{Y \subset S} [d_Y (\mu(v_1), \mu(v_2))]_{\kappa_1}\\
        & \asymp_{K_5(\kappa_1), C_5(\kappa_1)} d_{\M} (\mu(v_1), \mu(v_2)).\\
    \end{split}
\end{equation*}

The second line is by Lemma 5.2 for some $\kappa_0 > 0$. The third is by Lemma 2.1 since $d_T (v_1, v_2) \asymp_{K_0,C_0} d_S (\phi(v_1), \phi(v_2))$, assuming $\kappa_0 > 2K_0 C_0$. The fourth is by equation (6) assuming $\kappa_0 > \kappa_4$. The fifth is by Lemmas 2.1 and 5.9 assuming $C_4 > 4C_3.$ The sixth is taking $\kappa_1 = min \{ C_0, \frac{2C_3}{\xi(S)}\}.$ Finally, the last comparability is by Theorem 4.3.
\end{proof}

In fact there is even more information about the geometry of $G$ in $\Mod$. Let $\mathcal{A}$ be the collection of all the component curves of all multicurves $A(w)$ and let $\MG$ denote the $\mu$-image of $V$ in $\M$.

\begin{thm}
There exists some $M_1 > 0$ such that for all domains $Y \neq Y_\alpha, S$ with $\alpha \in \mathcal{A}$ and for all $\mu_1, \mu_2 \in \MG$ we have

\begin{equation*}
    d_Y (\mu_1, \mu_2) \leq M_1.
\end{equation*}

\end{thm}

In full detail the following proof is slightly technical, so we give a brief description of the main idea. Since $\mu_1, \mu_2 \in \MG,$ there are vertices $v_1, v_2 \in T$ such that $\mu(v_i) = \mu_i$ for $i = 1, 2.$ The generic case is when the boundary of $Y$ is far from the geodesic $[\phi(v_1) ,\phi(v_2)]$ and the bound follows from Theorem 2.2. When Theorem 2.2 is not applicable there is some segment $\gamma_Y$ of $[\phi(v_1) ,\phi(v_2)]$ which has empty projection to $Y.$ The general situation to consider here is when $\partial Y$ lies somewhere in the middle of $[\phi(v_1), \phi(v_2)].$

We describe an explicit path in $\C$ from $\phi(v_1)$ to $\phi(v_2)$ that has nonempty bounded projection to $Y.$ See figure 9. Starting at $\phi(v_1)$, we move along $[\phi(v_1), \phi(v_2)]$ and get uniformly close to $\gamma_Y$. Then we move along a geodesic to a vertex $\phi(w_l)$ in $\phi([v_1, v_2]_T),$ the quasigeodesic image of the geodesic in $T$ connecting $v_1$ and $v_2$. By Theorem 2.2 the concatenation of these two paths has diameter no more than $2M$ in $\CY$. Similarly there is a path from $\phi(v_2)$ to some vertex $\phi(w_m)$ in $\phi([v_1, v_2]_T)$ with $\CY$-diameter no more than $2M$. Finally, we use the $\phi$-quasiisometric embedding constants to bound the length of $\phi([w_l, w_m]_T),$ which is the segment of $\phi([v_1, v_2]_T)$ connecting $\phi(w_l)$ and $\phi(w_m),$ and use Lemma 5.7 to bound the projection of this segment to $\CY$.

The concern one should have is whether the path we've described has vertices with empty projection to $Y.$ By our choice of ``close'' to $\gamma_Y$, we can assure the paths from $\phi(v_1)$ to $\phi(w_l)$ and from  $\phi(w_m)$ to $\phi(v_2)$ have nonempty projection. Finally, for the part connecting $\phi(w_l)$ to $\phi(w_m)$, we describe in the proof why one can simply delete or replace the segments with empty projection from a sum that bounds the distance.

Our last remark is that we make repeated use of the constants $M, 2,$ and $k(R)$ from Theorem 2.2, Theorem 2.3, and Lemma 4.2, respectively, throughout the following proof without stating the respective result each time.

\begin{figure}[h]
    \centering
    \includegraphics[trim= 4.5cm 10cm 1.7cm 5.8cm, clip=true, scale = 1]{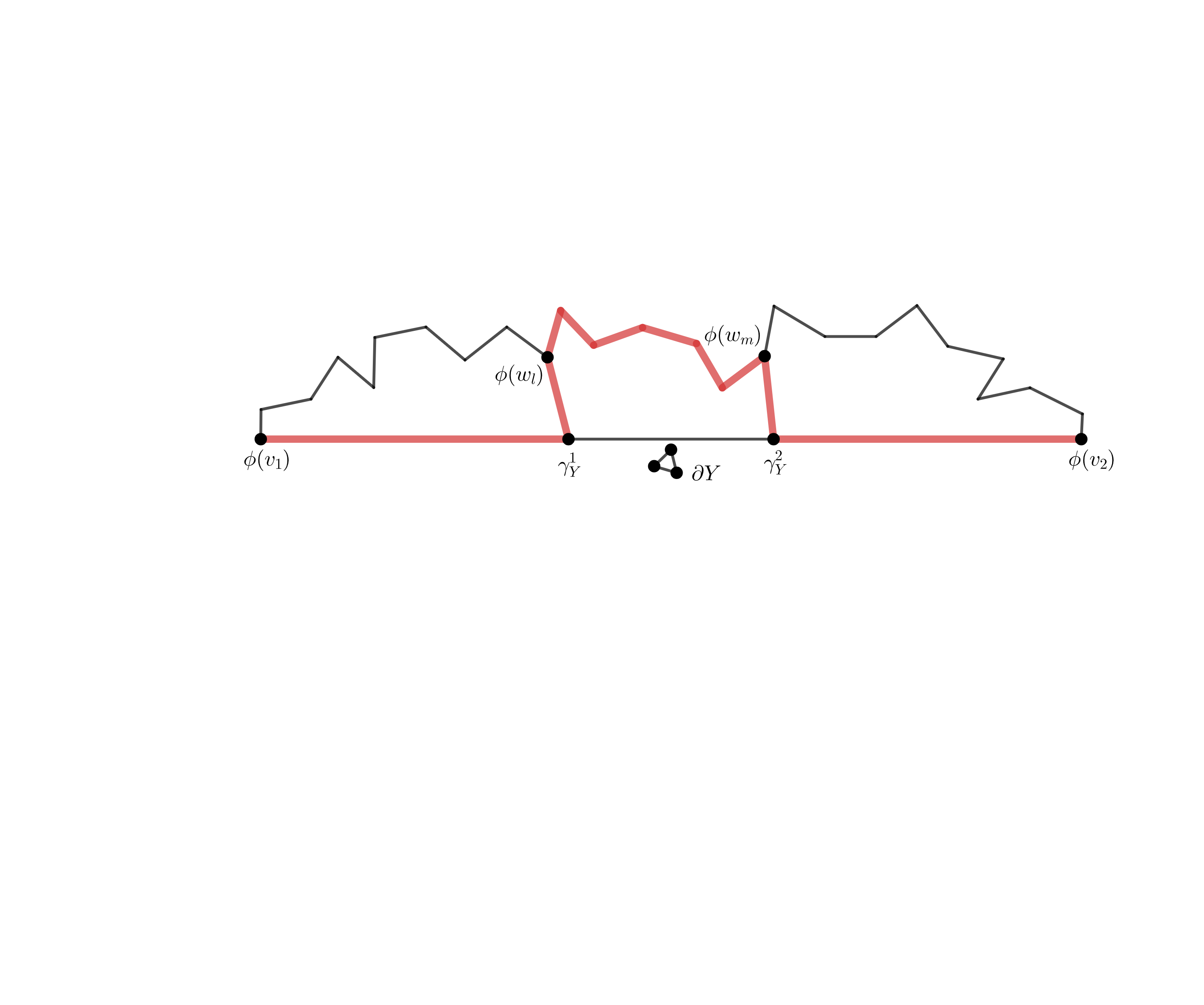}
    \caption{The path}
    \label{fig: the path}
\end{figure}

\begin{proof}
We continue using the same notation as above.

There are two main cases. The first is when $\partial Y$ is distance greater than 1 from $[\phi(v_1), \phi(v_2)]$ in $\C$ and hence $d_Y(\mu, \nu) \leq M + 2k(R)$. The second is when  $\partial Y$ is within distance $1$ of $[\phi(v_1), \phi(v_2)]$ and hence Theorem 2.2 is not applicable. 

Label the vertices of $[\phi(v_1), \phi(v_2)]$ traversed in order by $\phi(v_1), \gamma_1 ..., \gamma_N, \phi(v_2)$. That is

\begin{equation*}
    [\phi(v_1), \phi(v_2)] = [\phi(v_1), \gamma_1] \cup [\gamma_1, \gamma_2] \cup ... \cup [\gamma_N, \phi(v_2)].
\end{equation*}

Considering the path of the geodesic $[v_1, v_2]_T$ in $T,$ label vertices of $W$ traversed in order by $w_1, ..., w_{N'}$. That is
\begin{equation*}
   [v_1, v_2]_T = [v_1, w_1]_T \cup [w_1, w_2]_T \cup ... \cup [w_{N'}, v_2]_T.
\end{equation*}

Recall that
\begin{equation*}
   \phi([v_1, v_2]_T) = [\phi(v_1), \phi(w_1)] \cup [\phi(w_1), \phi(w_2)] \cup ... \cup [\phi(w_{N'}), \phi(v_2)].
\end{equation*}

Either one, two, or three adjacent vertices of $[\phi(v_1), \phi(v_2)]$ have empty projection to $Y.$ Let $\gamma_Y$ denote this segment with empty projection. Going from $\phi(v_1)$ to $\phi(v_2)$, let $\gamma_Y^1, \gamma_Y^2$ denote the first, last point of $\gamma_Y$, respectively.
 
Recall that $R_0(K_0, C_0, \delta)$ is the stability constant and $D = d_S(A,B) \geq 1$ since $G$ is a nontrivial free product.

\textbf{Case 1}: $d_S(\phi(v_1), \gamma_Y^1), d_S(\gamma_Y^2, \phi(v_2)) > R_0 +5D.$

Let $1 \leq j < k \leq N$ be such that $d_S(\gamma_j, \gamma_Y^1) = d_S(\gamma_Y^2, \gamma_k) = R_0 + 5D$. Let $\pi:[\phi(v_1), \phi(v_2)] \to \phi([v_1, v_2]_T)$ be  a closest point projection. Let $1\leq l \leq m \leq N'$ be such that $\pi(\gamma_j) \in [\phi(w_{l-1}), \phi(w_l)], \pi(\gamma_\nu) \in [\phi(w_m), \phi(w_{m+1})]$. Notice  $d_S(\gamma_j, \phi(w_l)), d_S(\gamma_k, \phi(w_m)) \leq R_0 + D$. By our choice of constants it's easy to check that for all $s$ in the connected paths $[\phi(v_1), \gamma_j] \cup [\gamma_j, \phi(w_l)]$ and $[\phi(w_m), \gamma_k] \cup [\gamma_k, \phi(v_2)]$

\begin{equation*}
    \pi_Y(s) \neq \emptyset.
\end{equation*}

It also follows from our choice of constants that $\pi_Y(A(w_l)), \pi_Y(A(w_m)) \neq \emptyset$. Then we have

\begin{equation*}
    d_Y(\mu(v_1), A(w_l)), d_Y(A(w_m), \mu(v_2)) \leq 2M + k(R) + 2.
\end{equation*}

Then we are done if we can show  $d_Y(A(w_l), A(w_m))$ is uniformly bounded. If $\pi_Y(A(w_n)) \neq \emptyset$ for all $l < n < m$ then by the triangle inequality we have

\begin{equation}
    \sum_{n=l}^{m-1} d_Y(A(w_n), A(w_{n+1})) \leq C_1(m-l).
\end{equation}

Suppose that there is some $l < p < m$ such that $\pi_Y(A(w_p)) = \emptyset$. Then $Y \subset S - A(w_p)$. Let $H(w_p)$ denote the coset associated to $A(w_p)$. Since $G$ is PGF, $\pi_Y(A(w_{p \pm 1})) \neq \emptyset$ since $A(w_p) \cup A(w_{p \pm 1})$ is a marking by Lemma 5.3. Notice that $A(w_{p-1})$ and $A(w_{p+1})$ differ by an element of $H(w_p)$. That is, there is some $t \in H(w_p)$ such that $t(A(w_{p-1})) = A(w_{p+1})$. Since $t$ is some product of twists on $A(w_p)$ and $Y \subset S - A(w_p)$, $t$ acts trivially on $\CY$. Hence $\pi_Y(A(w_{p-1})) = \pi_Y(A(w_{p+1}))$. Also notice that $\pi_Y(A(w_p)) = \emptyset$ implies $\pi_Y(A(w_q)) \neq \emptyset$ for all other $l < q < m$ distinct from $p.$ This follows from Lemma 5.3.

By the above, if $l+1 < p < m-1$ and $\pi_Y(A(w_p)) = \emptyset$, the following sum is well-defined and bounds $d_Y(A(w_l), A(w_m))$ by the triangle inequality

\begin{multline*}
    d_Y(A(w_l), A(w_{l+1})) + ... + d_Y(A(w_{p - 2}), A(w_{p - 1})) + d_Y(A(w_{p + 1}), A(w_{p + 2})) + ... \\ + d_Y(A(w_{m-1}), A(w_m)) \leq C_1(m-l).
\end{multline*}

We've simply deleted from (7) the two terms containing $A(w_{p})$. Notice we've excluded the cases $\pi_Y(A(w_{l+1})) = \emptyset$ or $\pi_Y(A(w_{m-1})) = \emptyset$. Say $\pi(A(w_{l+1})) = \emptyset$. The two terms in (7) containing $A(w_{l+1})$ are  $d_Y(A(w_l), A(w_{l+1})), d_Y(A(w_{l+1}), A(w_{l+2}))$. Deleting the first term renders the sum useless in bounding $d_Y(A(w_l), A(w_m)).$ So instead replace these two terms and $d_Y(A(w_{l+2}), A(w_{l+3}))$ with $d_Y(A(w_l), A(w_{l+3})) = d_Y(A(w_{l+2}), A(w_{l+3})) \leq C_1$. It should now be clear that, even if there is $l <  p < m$ with $\pi_Y(A(w_p)) = \emptyset$, by deleting or replacing terms from (7), there is a well-defined sum showing that

\begin{equation*}
    d_Y(A(w_l), A(w_m)) \leq C_1(m - l)
\end{equation*}

Finally, we note that $m-l$ is uniformly bounded above by a constant depending only on $K_0, C_0, R_0,$ and $D$. Specifically, $m-l \leq K_0 d_S(\phi(w_l), \phi(w_m)) + K_0 C_0 \leq  K_0(4R_0 + 12D) + K_0 C_0.$

\textbf{Case 2}: $d_S(\phi(v_1), \gamma_Y^1), d_S(\gamma_Y^2, \phi(v_2)) \leq R_0 + 5D.$

If $\pi_Y(A(w_n)) \neq \emptyset$ for all $1 \leq n \leq N'$ then the following sum is well-defined and bounds $d_Y(\mu, \nu)$

\begin{equation}
        d_Y(\mu_1, A(w_1)) + \sum_{n=1}^{N' - 1} d_Y(A(w_n), A(w_{n+1})) + d_Y(A(w_{N'}), \mu_2) \leq 2C_2 + C_1(N' - 1)
\end{equation}

Similar to case 1, if there is some $1<p<N'$ such that $\pi_Y(A(w_p)) = \emptyset$, simply delete the two terms in equation (8) that contain $A(w_p)$ to get a smaller sum.

We've excluded two cases: when $\pi_Y(w_1) = \emptyset$ or $\pi_Y(w_{N'}) = \emptyset$. Say $\pi_Y(w_1) = \emptyset$. Then the two terms in (8) containing $A(w_1)$ are $d_Y(\mu(v_1), A(w_1)), d_Y(A(w_1), A(w_2))$. Deleting the first term from (8) renders the sum useless in bounding $d_Y(\mu(v_1), \mu(v_2))$. So instead replace the two terms with $d_Y(\mu(v_1), A(w_2)) = d_Y(\mu(v_1), A(w_0)) \leq C_2$ where $w_0$ is the other endpoint of the edge $[w_0, w_1]$ containing $v_1$ in $T$. Again, $\pi_Y(A(w_0)) = \pi_Y(A(w_2))$ since $\pi_Y(A(w_1)) = \emptyset$.

Similarly define $w_{N' + 1}$ if $\pi_Y(w_N) = \emptyset$. It should now be clear that, even if there exists $1\leq p \leq N'$ with $\pi_Y(A(w_p)) = \emptyset$, by deleting or replacing terms from (8), there is a well-defined sum showing that

\begin{equation*}
    d_Y(\mu_1, \mu_2) \leq 2C_2 + C_1(N' - 1)
\end{equation*}

Finally, we note that $N'$ is uniformly bounded above by a constant depending only on $K_0, C_0, R_0,$ and $D.$

\textbf{Case 3}: $d_S(\phi(v_1), \gamma_Y^1) \leq R_0 + 5D$ and $d_S(\gamma_Y^2, \phi(v_2)) > R_0 + 5D$, or $d_S(\phi(v_1), \gamma_Y^1) > R_0 + 5D$ and $d_S(\gamma_Y^2, \phi(v_2)) \leq R_0 + 5D.$

In words,  this is the case when $\partial Y$ is close to either $\phi(v_1)$ or $\phi(v_2)$. It's easy to see how to combine the techniques of cases 1 and 2 to get a bound for this case.

Finally, take $M_1$ to be the maximum bound from all three cases.
\end{proof}

\section{Another Notion of Geometric Finiteness}








Since $G= H_A * H_B \cong \Z^n * \Z^m$ is a right angled-Artin group, it is a \textit{hierarchically hyperbolic space} \cite{BHS1}. Moreover, the inclusion of certain $G$ into $\Mod$ is a \textit{hieromorphism} \cite[Definition 1.10]{DHS}. Namely, when the generators of $H_A, H_B$ are all powers of Dehn twists (that is, they are not multitwists) the inclusion is a hieromorphism. We refer the reader to \cite[Definition 1.10]{DHS} for the terminology and precise definition, and sketch how our groups satisfy the definition.

Using the CAT(0) cube complex $\widetilde{X}$ as a model for $G$ and the general marking graph $\M$ as a model for $\Mod$, we let $f: \widetilde{X} \to \M$ be the coarse map $\mu: \widetilde{X} \to \M$ defined in section 5. We take $\widetilde{X},$ flats, and ``subflats'' as our index set (see \cite[Proposition 8.3, Remark 13.2]{BHS1}). Given any flat, an axis of the flat corresponds to twisting in a curve in $\mathcal{A},$ the $G$-orbit of $A \cup B$. This correspondence partially defines the map $\pi(f)$ between index sets. For flats $F$ of dimension greater than 1, let $\pi(f)(F)$ be the disjoint union of the curves corresponding to its underlying axes. Given an axis $U$ in $\widetilde{X}$ with points corresponding to elements say $gt^n,$ with $g \in G$ and $n \in \Z,$ its associated hyperbolic space $\mathcal{C}U$ is a line subdivided into intervals by ``midpoints,'' say $gt^\frac{2n+1}{2},$ corresponding to hyperplanes. A quasiisometric embedding $\rho(f, U)$ may be given by

\begin{equation*}
    \rho(f,U)(gt^\frac{2n+1}{2}) = \pi_{\pi(f)(U)} (gt^n(\mu)) \subset \mathcal{C}_{\pi(f)(U)}.
\end{equation*}

For flats $U \subset \widetilde{X}$ with dimension greater than 1, the associated hyperbolic space $\mathcal{C}U$ is a join of the associated hyperbolic spaces of its underlying axes, hence has uniformly bounded diameter. So one can simply choose $\rho(f, U)$ to be a constant map to $\mathcal{C}\pi(f)(U)$. Coarse commutativity of the diagrams from \cite[Definition 1.10]{DHS} follows from Lemma 5.9.

Theorem 5.10 tells us that these hieromorphisms are in fact an \textit{extensible}
\cite[Definition 5.5]{DHS}. This implies that such $G = H_A * H_B$ are also geometrically finite in the sense of \cite[Definition 2]{DHS}.

There is a more general notion of a  a \textit{slanted} hieromorphism \cite[Definition 5.1]{DHS}. This allows us to consider $H_A, H_B$ with multitwist generators. It is easy to see that if the generators of $H_A$ have disjoint supports, and similarly for $H_B,$ then the inclusion of $G$ into $\Mod$ is an extensible slanted hieromorphism, hence is geometrically finite in the sense of \cite{DHS}.

But in general, when the generators of $H_A$ (or $H_B$) do not have disjoint supports, it is not always clear how define $\pi(f)$.  Defining $\pi(f)$ in the obvious way generally fails to satisfy item (I) of \cite[Definition 5.1]{DHS}. It is an interesting question whether the definition of a slanted hieromorphism can be relaxed to include these examples yet still lead to geometric finiteness in the sense of \cite{DHS}. Alternatively, one might look for a different HHS structure on $\widetilde{X}$ and definition of $\pi(f)$ for these general examples in order to satisfy the current definition of slanted hieromorphisms.

\printbibliography

\end{document}